\documentclass[reqno]{amsart}

\usepackage{pstricks,amsmath,amsthm,amssymb,amsfonts,mathtools,amscd}
\numberwithin{equation}{section} 
\theoremstyle{plain}
\newtheorem{theo+}           {Theorem}      [section]
\newtheorem{prop+}  [theo+]  {Proposition}
\newtheorem{coro+}  [theo+]  {Corollary}
\newtheorem{lemm+}  [theo+]  {Lemma}
\newtheorem{defi+}  [theo+]  {Definition}

\theoremstyle{definition}
\newtheorem{rema+}  [theo+]  {Remark}
\newtheorem{prob+}  [theo+]  {Problem}
\newtheorem{exam+}  [theo+]  {Example}

\newenvironment{theorem}{\begin{theo+}}{\end{theo+}}
\newenvironment{proposition}{\begin{prop+}}{\end{prop+}}
\newenvironment{corollary}{\begin{coro+}}{\end{coro+}}
\newenvironment{lemma}{\begin{lemm+}}{\end{lemm+}}

\newenvironment{texdraw}{\leavevmode\btexdraw}{\etexdraw}

\input texdraw

\catcode`\@=11

\thinlines
\newskip\Einheit \Einheit=0.6cm
\newcount\xcoord \newcount\ycoord
\newdimen\xdim \newdimen\ydim \newdimen\PfadD@cke \newdimen\Pfadd@cke
\def\PfadDicke#1{\PfadD@cke#1 \divide\PfadD@cke by2 \Pfadd@cke\PfadD@cke \multiply\PfadD@cke by2}
\long\def\LOOP#1\REPEAT{\def\BODY{#1}\ITERATE}
\def\ITERATE{\BODY \let\next\ITERATE \else\let\next\relax\fi \next}
\let\REPEAT=\fi
\def\Punkt{\hbox{\raise-2pt\hbox to0pt{\hss\scriptsize$\bullet$\hss}}}
\def\DuennPunkt(#1,#2){\unskip
  \raise#2 \Einheit\hbox to0pt{\hskip#1 \Einheit
          \raise-2.5pt\hbox to0pt{\hss\normalsize$\bullet$\hss}\hss}}
\def\NormalPunkt(#1,#2){\unskip
  \raise#2 \Einheit\hbox to0pt{\hskip#1 \Einheit
          \raise-3pt\hbox to0pt{\hss\large$\bullet$\hss}\hss}}
\def\DickPunkt(#1,#2){\unskip
  \raise#2 \Einheit\hbox to0pt{\hskip#1 \Einheit
          \raise-4pt\hbox to0pt{\hss\Large$\bullet$\hss}\hss}}
\def\DickerPunkt(#1,#2){\unskip
  \raise#2 \Einheit\hbox to0pt{\hskip#1 \Einheit
          \raise-3.8pt\hbox to0pt{\hss\LARGE$\bullet$\hss}\hss}}
\def\Kreis(#1,#2){\unskip
  \raise#2 \Einheit\hbox to0pt{\hskip#1 \Einheit
          \raise-4pt\hbox to0pt{\hss\Large$\circ$\hss}\hss}}
 \def\Kreisbig(#1,#2){\unskip
   \raise#2 \Einheit\hbox to0pt{\hskip#1 \Einheit
           \raise-3.8pt\hbox to0pt{\hss\LARGE$\circ$\hss}\hss}}
\def\Diagonale(#1,#2)#3{\unskip\leavevmode
  \xcoord#1\relax \ycoord#2\relax
      \raise\ycoord \Einheit\hbox to0pt{\hskip\xcoord \Einheit
         \unitlength\Einheit
         \line(1,1){#3}\hss}}
\def\AntiDiagonale(#1,#2)#3{\unskip\leavevmode
  \xcoord#1\relax \ycoord#2\relax \advance\xcoord by -0.05\relax
      \raise\ycoord \Einheit\hbox to0pt{\hskip\xcoord \Einheit
         \unitlength\Einheit
         \line(1,-1){#3}\hss}}
\def\Pfad(#1,#2),#3\endPfad{\unskip\leavevmode
  \xcoord#1 \ycoord#2 \thicklines\ZeichnePfad#3\endPfad\thinlines}
\def\ZeichnePfad#1{\ifx#1\endPfad\let\next\relax
  \else\let\next\ZeichnePfad
    \ifnum#1=1
      \raise\ycoord \Einheit\hbox to0pt{\hskip\xcoord \Einheit
         \vrule height\Pfadd@cke width1 \Einheit depth\Pfadd@cke\hss}%
      \advance\xcoord by 1
    \else\ifnum#1=2
      \raise\ycoord \Einheit\hbox to0pt{\hskip\xcoord \Einheit
        \hbox{\hskip-1pt\vrule height1 \Einheit width\PfadD@cke depth0pt}\hss}%
      \advance\ycoord by 1
    \else\ifnum#1=3
      \raise\ycoord \Einheit\hbox to0pt{\hskip\xcoord \Einheit
         \unitlength\Einheit
         \line(1,1){1}\hss}
      \advance\xcoord by 1
      \advance\ycoord by 1
    \else\ifnum#1=4
      \raise\ycoord \Einheit\hbox to0pt{\hskip\xcoord \Einheit
         \unitlength\Einheit
         \line(1,-1){1}\hss}
      \advance\xcoord by 1
      \advance\ycoord by -1
    \fi\fi\fi\fi
  \fi\next}
\def\hSSchritt{\leavevmode\raise-.4pt\hbox to0pt{\hss.\hss}\hskip.2\Einheit
  \raise-.4pt\hbox to0pt{\hss.\hss}\hskip.2\Einheit
  \raise-.4pt\hbox to0pt{\hss.\hss}\hskip.2\Einheit
  \raise-.4pt\hbox to0pt{\hss.\hss}\hskip.2\Einheit
  \raise-.4pt\hbox to0pt{\hss.\hss}\hskip.2\Einheit}
\def\vSSchritt{\vbox{\baselineskip.2\Einheit\lineskiplimit0pt
\hbox{.}\hbox{.}\hbox{.}\hbox{.}\hbox{.}}}
\def\DSSchritt{\leavevmode\raise-.4pt\hbox to0pt{%
  \hbox to0pt{\hss.\hss}\hskip.2\Einheit
  \raise.2\Einheit\hbox to0pt{\hss.\hss}\hskip.2\Einheit
  \raise.4\Einheit\hbox to0pt{\hss.\hss}\hskip.2\Einheit
  \raise.6\Einheit\hbox to0pt{\hss.\hss}\hskip.2\Einheit
  \raise.8\Einheit\hbox to0pt{\hss.\hss}\hss}}
\def\dSSchritt{\leavevmode\raise-.4pt\hbox to0pt{%
  \hbox to0pt{\hss.\hss}\hskip.2\Einheit
  \raise-.2\Einheit\hbox to0pt{\hss.\hss}\hskip.2\Einheit
  \raise-.4\Einheit\hbox to0pt{\hss.\hss}\hskip.2\Einheit
  \raise-.6\Einheit\hbox to0pt{\hss.\hss}\hskip.2\Einheit
  \raise-.8\Einheit\hbox to0pt{\hss.\hss}\hss}}
\def\SPfad(#1,#2),#3\endSPfad{\unskip\leavevmode
  \xcoord#1 \ycoord#2 \ZeichneSPfad#3\endSPfad}
\def\ZeichneSPfad#1{\ifx#1\endSPfad\let\next\relax
  \else\let\next\ZeichneSPfad
    \ifnum#1=1
      \raise\ycoord \Einheit\hbox to0pt{\hskip\xcoord \Einheit
         \hSSchritt\hss}%
      \advance\xcoord by 1
    \else\ifnum#1=2
      \raise\ycoord \Einheit\hbox to0pt{\hskip\xcoord \Einheit
        \hbox{\hskip-2pt \vSSchritt}\hss}%
      \advance\ycoord by 1
    \else\ifnum#1=3
      \raise\ycoord \Einheit\hbox to0pt{\hskip\xcoord \Einheit
         \DSSchritt\hss}
      \advance\xcoord by 1
      \advance\ycoord by 1
    \else\ifnum#1=4
      \raise\ycoord \Einheit\hbox to0pt{\hskip\xcoord \Einheit
         \dSSchritt\hss}
      \advance\xcoord by 1
      \advance\ycoord by -1
    \fi\fi\fi\fi
  \fi\next}
\def\Koordinatenachsen(#1,#2){\unskip
 \hbox to0pt{\hskip-.5pt\vrule height#2 \Einheit width.5pt depth1 \Einheit}%
 \hbox to0pt{\hskip-1 \Einheit \xcoord#1 \advance\xcoord by1
    \vrule height0.25pt width\xcoord \Einheit depth0.25pt\hss}}
\def\Koordinatenachsen(#1,#2)(#3,#4){\unskip
 \hbox to0pt{\hskip-.5pt \ycoord-#4 \advance\ycoord by1
    \vrule height#2 \Einheit width.5pt depth\ycoord \Einheit}%
 \hbox to0pt{\hskip-1 \Einheit \hskip#3\Einheit 
    \xcoord#1 \advance\xcoord by1 \advance\xcoord by-#3 
    \vrule height0.25pt width\xcoord \Einheit depth0.25pt\hss}}
\def\Gitter(#1,#2){\unskip \xcoord0 \ycoord0 \leavevmode
  \LOOP\ifnum\ycoord<#2
    \loop\ifnum\xcoord<#1
      \raise\ycoord \Einheit\hbox to0pt{\hskip\xcoord \Einheit\Punkt\hss}%
      \advance\xcoord by1
    \repeat
    \xcoord0
    \advance\ycoord by1
  \REPEAT}
\def\Gitter(#1,#2)(#3,#4){\unskip \xcoord#3 \ycoord#4 \leavevmode
  \LOOP\ifnum\ycoord<#2
    \loop\ifnum\xcoord<#1
      \raise\ycoord \Einheit\hbox to0pt{\hskip\xcoord \Einheit\Punkt\hss}%
      \advance\xcoord by1
    \repeat
    \xcoord#3
    \advance\ycoord by1
  \REPEAT}
\def\Label#1#2(#3,#4){\unskip \xdim#3 \Einheit \ydim#4 \Einheit
  \def\lo{\advance\xdim by-.5 \Einheit \advance\ydim by.5 \Einheit}%
  \def\llo{\advance\xdim by-.25cm \advance\ydim by.5 \Einheit}%
  \def\loo{\advance\xdim by-.5 \Einheit \advance\ydim by.25cm}%
  \def\o{\advance\ydim by.25cm}%
  \def\ro{\advance\xdim by.5 \Einheit \advance\ydim by.5 \Einheit}%
  \def\rro{\advance\xdim by.25cm \advance\ydim by.5 \Einheit}%
  \def\roo{\advance\xdim by.5 \Einheit \advance\ydim by.25cm}%
  \def\l{\advance\xdim by-.30cm}%
  \def\r{\advance\xdim by.30cm}%
  \def\lu{\advance\xdim by-.5 \Einheit \advance\ydim by-.6 \Einheit}%
  \def\llu{\advance\xdim by-.25cm \advance\ydim by-.6 \Einheit}%
  \def\luu{\advance\xdim by-.5 \Einheit \advance\ydim by-.30cm}%
  \def\u{\advance\ydim by-.30cm}%
  \def\ru{\advance\xdim by.5 \Einheit \advance\ydim by-.6 \Einheit}%
  \def\rru{\advance\xdim by.25cm \advance\ydim by-.6 \Einheit}%
  \def\ruu{\advance\xdim by.5 \Einheit \advance\ydim by-.30cm}%
  #1\raise\ydim\hbox to0pt{\hskip\xdim
     \vbox to0pt{\vss\hbox to0pt{\hss$#2$\hss}\vss}\hss}%
}
\catcode`\@=12

\def\ldreieck{\bsegment
  \rlvec(0.866025403784439 .5) \rlvec(0 -1)
  \rlvec(-0.866025403784439 .5)  
  \savepos(0.866025403784439 -.5)(*ex *ey)
        \esegment
  \move(*ex *ey)
        }
\def\rdreieck{\bsegment
  \rlvec(0.866025403784439 -.5) \rlvec(-0.866025403784439 -.5)  \rlvec(0 1)
  \savepos(0 -1)(*ex *ey)
        \esegment
  \move(*ex *ey)
        }
\def\rhombus{\bsegment
  \rlvec(0.866025403784439 .5) \rlvec(0.866025403784439 -.5) 
  \rlvec(-0.866025403784439 -.5)  \rlvec(0 1)        
  \rmove(0 -1)  \rlvec(-0.866025403784439 .5) 
  \savepos(0.866025403784439 -.5)(*ex *ey)
        \esegment
  \move(*ex *ey)
        }
\def\RhombusA{\bsegment
  \rlvec(0.866025403784439 .5) \rlvec(0.866025403784439 -.5) 
  \rlvec(-0.866025403784439 -.5) \rlvec(-0.866025403784439 .5) 
  \savepos(0.866025403784439 -.5)(*ex *ey)
        \esegment
  \move(*ex *ey)
        }
\def\RhombusB{\bsegment
  \rlvec(0.866025403784439 .5) \rlvec(0 -1)
  \rlvec(-0.866025403784439 -.5) \rlvec(0 1) 
  \savepos(0 -1)(*ex *ey)
        \esegment
  \move(*ex *ey)
        }
\def\RhombusC{\bsegment
  \rlvec(0.866025403784439 -.5) \rlvec(0 -1)
  \rlvec(-0.866025403784439 .5) \rlvec(0 1) 
  \savepos(0.866025403784439 -.5)(*ex *ey)
        \esegment
  \move(*ex *ey)
        }

\def\vdSchritt{\bsegment
  \lpatt(.05 .13)
  \rlvec(0 -1) 
  \savepos(0 -1)(*ex *ey)
        \esegment
  \move(*ex *ey)
        }

\def\odSchritt{\bsegment
  \lpatt(.05 .13)
  \rlvec(-0.866025403784439 -.5) 
  \savepos(-0.866025403784439 -.5)(*ex *ey)
        \esegment
  \move(*ex *ey)
        }

\def\ringerl(#1 #2){\move(#1 #2)\fcir f:0 r:.15}
\def\ringop(#1 #2){\move(#1 #2)\lcir r:.15}

\newcommand{\ti}{\textup i}
\newcommand{\sgn}{\operatorname{sgn}}

\makeatletter
\newcommand{\pushright}[1]{\ifmeasuring@#1\else\omit\hfill$\displaystyle#1$\fi\ignorespaces}
\makeatother

\begin{document}

\baselineskip 18pt
\larger[2]
\title
[Selberg integrals, Askey--Wilson polynomials and lozenge tilings]
{Selberg integrals, Askey--Wilson polynomials and \\
 lozenge tilings of a hexagon with a triangular hole} 
\author{Hjalmar Rosengren}
\address
{Department of Mathematical Sciences
\\ Chalmers University of Technology and G\"oteborg
 University\\SE-412~96 G\"oteborg, Sweden}
\email{hjalmar@chalmers.se}
\urladdr{http://www.math.chalmers.se/{\textasciitilde}hjalmar}
\subjclass[2000]{}

\begin{abstract}
We obtain an explicit formula for a certain weighted enumeration
of lozenge tilings of a hexagon with an arbitrary triangular hole. The complexity of
our expression depends on the distance from the hole to the center 
of the hexagon. This proves and  generalizes conjectures of Ciucu et al., who considered the case of plain enumeration when the triangle is located at or very near the center.
Our proof uses Askey--Wilson polynomials as a tool to relate discrete and continuous Selberg-type integrals.
\end{abstract}

\maketitle

\section{Introduction}  

 One of the most  influential 
 results of enumerative combinatorics is MacMahon's formula \cite{mm}
 $$\frac{H(a)H(b)H(c)H(a+b+c)}{H(a+b)H(a+c)H(b+c)} $$
 for the number of plane partitions contained in a box of size $a\times b\times c$, where $H(n)=\prod_{k=1}^n(k-1)!$. 
Equivalently, this identity enumerates
 lozenge tilings of a hexagon with side lengths $a$, $b$ and $c$. 

There has been quite
 a lot of work on lozenge 
tilings of a hexagon with various kinds of holes
\cite{c1,c2,c3,cf,cf2,ck1,ck2,ck3,ck4,e1,e2,hg,l1,l2,l3,ok,pro,pr}. 
In the seminal paper \cite{pro},
Propp  conjectured an explicit formula for the number of tilings of a hexagon $H$ whose side lengths are almost equal, with a small triangle $T$  removed from the center of $H$ (more  precisely, in the notation explained in \S \ref{wess}, this is the region $H\setminus T$ with $a=b=c$, $m=1$, $M=N=0$).
This conjecture was proved  in \cite{c1,hg}. 
More generally, Ciucu et al.\ \cite{cekz}  enumerated the tilings when the side lengths of $H$ and $T$ are arbitrary, but $T$ is still  positioned at (or very near) the center of $H$. They also conjectured enumerations
for some adjacent positions of $T$.

In the present paper, we consider  the general case, when the position of $T$ within $H$ is arbitrary.  Our main result, Theorem \ref{mt}, expresses a weighted extension of the number of tilings as a determinant, whose complexity depends  on the distance of $T$ from the center of $H$. Thus, it is a closed formed evaluation if the position of $T$ relative to the center of $H$ is fixed, but the side lengths of $T$ and $H$ are arbitrary.

As in \cite{cekz}, the starting point of our proof is the
 Gessel--Viennot method \cite{gv}, which gives an explicit determinant formula for the weighted enumeration. In \cite{cekz} the determinant is computed using the method of identification of factors \cite{kad}. It seems very difficult to handle
the more general determinants that we encounter in this way.
Instead, we derive a chain of intermediate expressions for our weighted enumeration as indicated in the following diagram.

$$\begin{CD}
\text{Weighted enumeration}@>\text{Gessel--Viennot}>\S\ref{gvs}>\text{Determinant I}\\
@. @V\text{minor expansion}V\S\ref{dss} V\\
\text{Determinant II}@<\text{Cauchy--Binet}<\S\ref{aws}<\text{Discrete Selberg integral}\\
@V\S\ref{css}V\text{Christoffel--Heine}V @.\\
\text{Continuous Selberg integral}@>\text{Christoffel--Heine}>\S\ref{css}>\text{Determinant III}
\end{CD}$$

\vspace*{1mm}
Here, Determinant I is obtained  by the Gessel--Viennot method. It has completely factored entries and its dimension is equal to one of the side lengths of $H$. Applying an appropriate minor expansion leads to a multivariable basic hypergeometric series. As it contains the factor $\prod_{i<j}(q^{m_j}-q^{m_i})^2$, with $m_j$ being summation indices, it can be considered as a discrete Selberg-type integral \cite{fw}. Special cases of this sum appear in \cite{cekz}, but are considered there as  consequences of the enumeration rather than as a tool.

In general, Selberg-type refers to hypergeometric series or integrals containing factors like $\prod_{i<j}|x_j-x_i|^c$, where the archetypal example is the integral 
$$\int_{[0,1]^n}\prod_{1\leq i<j\leq n}|x_j-x_i|^c\prod_{j=1}^nx_j^{a-1}(1-x_j)^{b-1}\,dx_j. $$ 
The cases $c=1$ and $c=2$ are determinantal, in the sense that they can be 
expressed as determinants of one-variable integrals.
In our setting, an application of 
 the Cauchy--Binet identity leads to  an
alternative determinant formula for the weighted enumeration, Determinant II. 
It is quite different from Determinant I as its entries are 
Askey--Wilson polynomials and its dimension is equal to the side length
of $T$. Using classical results on orthogonal polynomials due to Christoffel and Heine, we can rewrite Determinant II as a continuous Selberg integral, where
$\prod_{i<j}(x_j-x_i)^2$ is integrated against the Askey--Wilson orthogonality measure. The key observation is now  that the 
results of Christoffel and Heine can be applied in a different way to the same 
Selberg integral. This leads to our end result, Determinant III. 
Here, the matrix entries are again Askey--Wilson polynomials, but in base $q^2$ rather than $q$. The size of the determinant is related to the distance from $T$ to the center of $H$.  

The above decription of our proof is not quite accurate, as we glossed over two important technical aspects. First, Determinant I only applies when the side length $m$ of $T$ is even. To extend our result to odd $m$, we 
need an a priori result on how our weighted enumeration behaves as a function of $m$, Lemma \ref{pl}. This is achieved by another application of the 
Gessel--Viennot method. Second, the orthogonality relation for
 Askey--Wilson polynomials is actually  not valid for the specific parameters
appearing from the tiling problem. Thus, the continuous Selberg integral mentioned above does not make sense as an expression for the weighted enumeration, but only appears after continuation to a different range of parameters.

The explicit expression given in Theorem \ref{mt} is admittedly rather complicated, but we believe that 
the method of proof is  more interesting than the result. 
It seems  likely that there are other problems related to tilings and plane partitions that can be approached with similar methods.
For instance, one could ask for a ``dual'' of our result in the sense of \cite{ck3}.

{\bf Acknowledgements:} This research was partially supported by the Swedish Science Research Council.  All figures are based on packages
written by Theresia Eisenk\"olbl and Christian Krattenthaler.

\section{Main result}

\subsection{Weighted enumeration of tilings}
\label{wess}

Consider the triangular lattice in the plane, formed by equilateral triangles
of side length $1$ and height $\phi=\sqrt 3/2$. 
On this lattice, we draw a convex hexagon $H$ and remove
an equilateral  triangle $T\subseteq H$. We are interested in tilings of  $H\setminus T$ by lozenges, that is, by
 quadrilaterals formed by adjoining two adjacent lattice triangles.
Using the  bijection to lattice paths discussed in \S \ref{gvs}, it is easy to see that, for such tilings to exist, 
 $H$ must have consecutive side-lengths $a$, $b+m$, $c$, $a+m$, $b$, $c+m$,
where $m$ is the side-length of $T$. 
Moreover, the sides of $T$ must be 
 parallel to the long sides (of length $a+m$, $b+m$ and $c+m$)  
of $H$.

We will refer to the sides of $H$ by the  expression for their length; for instance, the side $b+m$ is the second side in the ordering given above. We will  picture the region $H\setminus T$ as in Figure \ref{tiling}. 
This  allows us to use terminology such as ``horizontal'' to 
refer to the direction orthogonal to the side $c$.

\begin{figure}
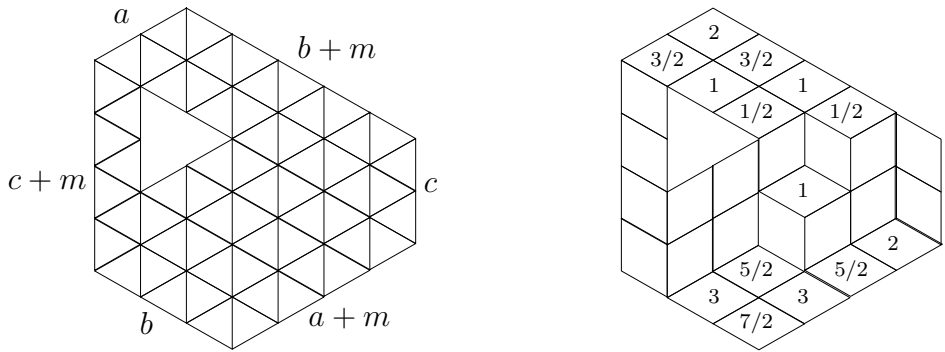
 
\centertexdraw{
\drawdim cm \setunitscale.7
\linewd.01
\rhombus \rhombus \rhombus \rhombus \rhombus 
\ldreieck
 \move(-0.866025 -.5)
 \rhombus \rhombus \rhombus \rhombus \rhombus \rhombus 
\ldreieck
 \move(-0.866025 -0.5)
\rdreieck\ldreieck 
 \rmove(0.866025 -0.5)
  \rhombus \rhombus \rhombus \rhombus 
 \move(-.866025 -1.5)
 \rdreieck\ldreieck 
  \rhombus \rhombus \rhombus \rhombus 
  \move(-.866025 -2.5)
  \rdreieck \rhombus \rhombus \rhombus \rhombus
  \move(-.866025 -3.5)
 \rdreieck \rhombus \rhombus \rhombus

 \move(10 0)
 \RhombusA  \RhombusA \RhombusA \RhombusA \RhombusB \RhombusB \RhombusA
 \move(9.133975 -.5)
  \RhombusA \RhombusA \RhombusA \RhombusB \RhombusA \RhombusB 
  \RhombusA 
  \move(10.866025 -2.5)
\RhombusB \RhombusB \RhombusA \RhombusA
\move(10 -3)
\RhombusB \RhombusB \RhombusA \RhombusA
 \move(9.133975 -0.5) \RhombusC 
 \move(9.133975 -1.5) \RhombusC 
 \move(9.133975 -2.5) \RhombusC 
 \move(9.133975 -3.5) \RhombusC 
 \move(14.330127 -1.5) \RhombusC 
\move(14.330127 -2.5) \RhombusC 
\htext(10.8 -0.1){$\scriptstyle 2$}
\htext(9.7 -0.7){$\scriptstyle{3/2}$}
\htext(11.4 -0.7){$\scriptstyle{3/2}$}
\htext(10.8 -1.1){$\scriptstyle 1$}
\htext(12.5 -1.1){$\scriptstyle 1$}
\htext(11.4 -1.7){$\scriptstyle{1/2}$}
\htext(13.1 -1.7){$\scriptstyle{1/2}$}
\htext(12.5 -3.1){$\scriptstyle 1$}
\htext(14.2 -4.1){$\scriptstyle 2$}
\htext(13.2 -4.7){$\scriptstyle 5/2$}
\htext(11.4 -4.7){$\scriptstyle 5/2$}
\htext(10.8 -5.1){$\scriptstyle 3$}
\htext(12.5 -5.1){$\scriptstyle 3$}
\htext(11.4 -5.7){$\scriptstyle 7/2$}
 \htext(-0.5 0.2){$a$}
 \htext(3.0 -0.5){$b+m$}
 \htext(5.4 -3.0){$c$}
 \htext(-2.5 -3){$c+m$}
 \htext(0 -5.7){$b$}
 \htext(3.2 -5.6){$a+m$}
}
\caption{The region $H\setminus T$ described by $(a,b,c,m,A,B,C)=(2,3,2,2,4,2,1)$. A tiling of this region,  with each horizontal lozenge labelled by its height.}
\label{tiling}
\end{figure}

To specify the position of $T$ within  $H$, let $A\phi$, $B\phi$ and $C\phi$ denote the distance from $T$ to the line containing the side  $a+m$, $b+m$ and $c+m$, respectively. It is easy to see that
\begin{equation}\label{vr}A+B+C=a+b+c, \end{equation}
for instance, by applying Viviani's theorem to the  triangle  formed by extending  the short sides of $H$.
The distances from $T$ to the lines containing the short sides
of $H$ are $(b+c-A)\phi$, $(a+c-B)\phi$ and $(a+b-C)\phi$. Thus, 
\begin{equation}\label{ii}0\leq A\leq b+c,\qquad 0\leq B\leq a+c,\qquad 0\leq C\leq a+b. \end{equation}
Conversely, any non-negative integers $a$, $b$, $c$, $A$, $B$, $C$ and $m$
subject to \eqref{vr}  and \eqref{ii} describe a region $H\setminus T$. Note that we include degenerate cases when  some sides of $H$ have length zero, when no triangle is removed
($m=0$) or when $T$ touches the boundary of $H$.
We will also specify the location of $T$ by the coordinates
\begin{equation}\label{dc} M=2A-b-c,\qquad N=2B-a-c,\end{equation}
so that
$$A=\frac{b+c+M}{2},\quad B=\frac{a+c+N}{2},\quad C=\frac{a+b-M-N}2. $$
Then, $M$ and $N$ are integers of the same parity as $b+c$ and $a+c$, respectively, such that
$$|M|\leq b+c,\qquad |N|\leq a+c,\qquad |M+N|\leq a+b. $$
As an example, the region in Figure \ref{tiling} corresponds to $(M,N)=(3,0)$.

Note that $(M,N)=(0,0)$ corresponds to  $T$ being located at the center
of  $H$. This case, and a few other cases with $T$ nearly central, were 
studied in \cite{cekz}. To be precise, these authors enumerated the  tilings when $(M,N)$ equals $(0,0)$ and $(0,1)$ and conjectured enumerations when $(M,N)$ equals 
 $(0,2)$ and $(0,3)$. 
In the present paper, we
will explain how to prove these conjectures and  obtain analogous results for \emph{any} $M$ and $N$. 

More generally, we will consider a weighted enumeration of tilings. We define the height $h$ of a
horizontal tile $Q$ to be
 the vertical distance from the center of $Q$ to  the center of $T$
(see  Figure \ref{tiling}). Our main object of study is the partition function
\begin{equation}\label{z}Z(q)=\sum_{\text{tilings of } H\setminus T}\,\prod_{\text{horizontal tiles}}
\frac{q^{h}+q^{-h}}{2}. \end{equation}
In particular, $Z(1)$ is 
the total number of tilings.
Our weight function is a special case of weights introduce in \cite{bgr} for plane partitions and \cite{sc2} for lattice paths. 

Note that, when $m=0$, $Z(q)$ is different from the volume generating function
for plane partitions 
computed by MacMahon. Up to a power of $q$, the latter is equal to
$$\tilde Z(q)=\sum_{\text{tilings}}\,\prod_{\text{horizontal tiles}}q^{2\tilde h},$$
where $\tilde h$  is the vertical distance from the center of a tile
to the bottom corner of $H$. When $m>0$, the function
$Z$ behaves  better than $\tilde Z$, being given by  
completely factored expressions in 
situations (e.g.\ $(M,N)=(0,0)$) when such expressions exist for the enumeration problem.

\subsection{Notation}

We will write $\sgn(k)=1$ for $k\geq 0$ and $\sgn(k)=-1$ for $k<0$.
Recall also the standard notation \cite{gr} 
\begin{align*}(a;q)_n&=\prod_{j=0}^{n-1}(1-aq^j), \\
(a_1,\dots,a_m;q)_n&=(a_1;q)_n\dotsm (a_m;q)_n, \\
{}_{r+1}\phi_r\left(\begin{matrix}a_1,\dots,a_{r+1}\\b_1,\dots,b_r\end{matrix};q,x\right)&=\sum_{k=0}^\infty\frac{(a_1,\dots,a_{r+1};q)_k}{(q,b_1,\dots,b_r;q)_k}\,x^k.
\end{align*}

When $x=(x_1,\dots,x_m)$, we will write $\Delta(x)=\prod_{1\leq i<j\leq m}(x_j-x_i)$
and also use notation such as $\Delta(q^k)=\prod_{1\leq i<j\leq m}(q^{k_j}-q^{k_i})$.
We  introduce the multiple basic hypergeometric series
\begin{equation}\label{mhs}{}_{r+1}\phi_r^{(m)}\left(\begin{matrix}a_1,\dots,a_{r+1}\\b_1,\dots,b_r\end{matrix};q,x\right)=\sum_{0\leq k_1<k_2<\dots<k_m}\Delta(q^k)^2
\prod_{j=1}^m\frac{(a_1,\dots,a_{r+1};q)_{k_j}x^{k_j}}{(q,b_1,\dots,b_r;q)_{k_j}}.
\end{equation}
In view of the factor $\Delta(q^k)^2$, it
 can be thought of as a discrete Selberg-type integral.
 
We introduce the $q$-hyperfactorial
\begin{equation}\label{qsf}H_q(m)=\begin{cases}\prod_{j=1}^m\left(q^{-\frac j2}-q^{\frac j2}\right)^{m-j}, & m=0,1,2,\dots,\\
\prod_{j=1}^{m+\frac 12}\left(q^{\frac 14-\frac j2}-q^{\frac j2-\frac 14}\right)^{m+\frac 12-j}, & m=-1/2,1/2,3/2,\dots.
\end{cases} \end{equation}
Equivalently,
$$H_q(m)=\begin{cases}q^{-\frac 12\binom{m+1}3}\prod_{j=1}^m(q;q)_{j-1}, & m=0,1,2,\dots,\\
q^{-\frac 1{16}\binom{2m+1}3}\prod_{j=1}^{m+\frac 12}(q^{\frac 12};q)_{j-1}, & m=-1/2,1/2,3/2,\dots.
\end{cases} $$
Deleting the prefactor from these expressions, we will write
$$\tilde H_q(m)=\begin{cases}\prod_{j=1}^m(q;q)_{j-1}, & m=0,1,2,\dots,\\
\prod_{j=1}^{m+\frac 12}(q^{\frac 12};q)_{j-1}, & m=-1/2,1/2,3/2,\dots.
\end{cases} $$
We  use both notations since our main result is easier to state in terms of $H_q$, but for the proof it is often more convenient to work with $\tilde H_q$.

We will sometimes write $H_q^+=H_q$ and
$$H_q^-(m)=\frac{H_{q^2}(m)}{H_q(m)}=\begin{cases}
\prod_{j=1}^m\left(q^{-\frac j2}+q^{\frac j2}\right)^{m-j}, & m=0,1,2,\dots,\\
\prod_{j=1}^{m+\frac 12}\left(q^{\frac 14-\frac j2}+q^{\frac j2-\frac 14}\right)^{m+\frac 12-j}, & m=-1/2,1/2,3/2,\dots.
\end{cases} $$
Repeated arguments stands for a product; for instance,
$$H_q(a_1,\dots,a_m)=H_q(a_1)\dotsm H_q(a_m). $$
Similar notation will be used for $\tilde H_q$. We collect some useful facts about the functions $H_q$ and $\tilde H_q$
in an Appendix.

\subsection{Statement of main result}
Our main result is formulated in terms of   determinants 
\begin{equation}\label{qd}Q^{MNn}(\alpha,\beta,\gamma;q)=\det_{1\leq j,k\leq M+N}\left(Q^{MNn}_{jk}(\alpha,\beta,\gamma;q)\right),\end{equation}
labelled by non-negative integers $M$, $N$, $n$ and generic parameters
$\alpha$, $\beta$, $\gamma$. The matrix entries $Q^{MNn}_{jk}(\alpha,\beta,\gamma;q)$ are given for 
for   $1\leq k\leq M$ and $n+j$ odd by
\begin{subequations}\label{qme}
\begin{multline}\label{qma} \frac{(\alpha^2,\alpha^2\gamma^2;q^2)_{(n+j-1)/2}(\alpha^2\beta^2;q^2)_{k-1}}{q^{\frac14(n+j-1)(n+j-3)+\binom{k-1}2}\alpha^{n+j+k-2}\beta^{k-1}\gamma^{\frac 12(n+j-1)}}\\
\times\,{}_{4}\phi_3\left(\begin{matrix}q^{1-j-n},\alpha^2\beta^2\gamma^2 q^{n+j-3},\alpha^4q^{2k-2},q^{2-2k}\\
\alpha^2,\alpha^2\beta^2,\alpha^2\gamma^2\end{matrix};q^2,q^2\right)\end{multline}
and for  $1\leq k\leq M$ and $n+j$ even by
\begin{multline}\label{qmb}
\ (\alpha q^{k-1}-\alpha^{-1}q^{1-k})\frac{(\alpha^2q^2,\alpha^2\gamma^2;q^2)_{(n+j-2)/2}(\alpha^2\beta^2;q^2)_{k-1}}{q^{\frac14(n+j-2)^2+\binom{k-1}2}\alpha^{n+j+k-3}\beta^{k-1}\gamma^{\frac 12(n+j-2)}}\\
\times\,{}_{4}\phi_3\left(\begin{matrix}q^{2-j-n},\alpha^2\beta^2\gamma^2 q^{n+j-2},\alpha^4q^{2k-2},q^{2-2k}\\
\alpha^2q^2,\alpha^2\beta^2,\alpha^2\gamma^2\end{matrix};q^2,q^2\right).\end{multline}
\end{subequations}
For the remaining cases $(M+1\leq k\leq M+N)$, they are determined by
$$Q^{MNn}_{j,M+k}(\alpha,\beta,\gamma;q)=Q^{NMn}_{j,k}(\beta,\alpha,\gamma;q), \qquad 1\leq k\leq N. $$
Though the structure of this determinant may seem complicated, 
we will see in \S \ref{css} that it appears  naturally in the 
context of Askey--Wilson polynomials. 
It is easy to check that  
\begin{equation}\label{qdi}Q^{MNn}(\alpha,\beta,\gamma;q)=(-1)^{MN+n(M+N)}Q^{MNn}(\alpha^{-1},\beta^{-1},\gamma^{-1};q^{-1}). \end{equation}

The  matrix elements \eqref{qme} are Laurent polynomials in $\alpha$, $\beta$ and  $ \gamma$.
In particular, we may (and will) specialize these variables to points where the ${}_4\phi_3$-sums without the prefactor are singular.
Note also that each matrix entry 
is a sum of at most $\max(M,N)$ terms. 
Thus, the following result gives a closed form evaluation of
$Z(q)$ for  fixed  $M$ and $N$. 

\begin{theorem}\label{mt}
With $\varepsilon=\sgn(MN)$, we have
\begin{equation}\label{mti}Z(q)=C\,Q^{|M|,|N|,b}\big(q^{\frac 12(1-b-c-m-|M|)},-\varepsilon q^{\frac 12(1+a+c+m-|N|)},q^{\frac12(m+1)};q\big), \end{equation}
where
{\allowdisplaybreaks
\begin{align}
\nonumber C&=\frac{(-1)^{\binom{|N|}2}\varepsilon^{\binom{|M|}2+N(b+M)}}{2^{\frac 12m(a+b+M+N)+ab-\frac{a+b}2+\frac 12\max(|a-b|,|M-N|)}}
\frac{H_{q^2}\left(\frac m2\right)^2}{H_q(|M|,|N|)}\\
\nonumber &\quad\times H_{q^2}\left(\left[\frac a2\right],\left[\frac {a+1}2\right],\left[\frac a2\right]+\frac{m+1}2,\left[\frac {a+1}2\right]+\frac{m-1}2\right)\\
\nonumber &\quad\times H_{q^2}\left(\left[\frac b2\right],\left[\frac {b+1}2\right],\left[\frac b2\right]+\frac{m+1}2,\left[\frac {b+1}2\right]+\frac{m-1}2\right)\\
\nonumber &\quad\times\frac{ H_{q^2}\left(\big[\frac {c+|M|}2\big],\big[\frac {c+|M|+1}2\big],\big[\frac {c+|M|}2\big]+\frac{m+1}2,\big[\frac {c+|M|+1}2\big]+\frac{m-1}2\right)}{H_q\left(\frac{a+b-M-N}2,\frac{a+b+M+N}2+m\right)H_{q^2}\left(\frac{a+b-|M|-|N|+m+1}2,\frac{a+b+|M|+|N|+m-1}2\right)}\\
\nonumber &\quad\times \frac{H_{q^2}\left(\frac{a+c+N}{2}+m\right)}{H_{q^2}\left(\frac{a+c+N}2,\frac{a+c-|N|+m-1}2\right)H_{q^2}\left(\frac{a+c-|N|}2+m\right)^2H_{q^4}\left(\frac{a+c-|N|+m+1}2\right)}\\
\nonumber &\quad\times\frac{H_q^-(a+c+m)}{H_{q^2}^-\left(\frac{a+c-|N|+m}2,\frac{a+c+|N|+m-1}2,\frac{a+c+|N|+m}2\right)}\\
\nonumber &\quad\times \frac{H_{q^2}\left(\frac{b+c-M}{2}\right)}{H_{q^2}\left(\frac{b+c-M}2+m,\frac{b+c+|M|+m+1}2\right)H_{q^2}\left(\frac{b+c+|M|}2\right)^2 H_{q^4}\left(\frac{b+c+|M|+m-1}2\right)}\\
\nonumber &\quad\times \frac{H_q^-(b+c+m)}{H_{q^2}^-\left(\frac{b+c-|M|+m}2,\frac{b+c-|M|+m+1}2,\frac{b+c+|M|+m}2\right)}\frac 1{H_q^-\left(\left|\frac{a-b+M-N}2\right|,\left|\frac{a-b-M+N}2\right|\right)}\\
\nonumber &\quad\times H_{q^2}\left(\left[\frac{a+b+c-|N|}2\right ]+m,\left[\frac{a+b+c-|N|+1}2\right ]+m\right)\\
\nonumber &\quad\times H_{q^2}\left(\left[\frac{a+b+c-|N|}2\right ]+\frac{m+1}2,\left[\frac{a+b+c-|N|+1}2\right ]+\frac{m-1}2\right)\\
\label{mtc} &\quad\times\frac{H_q^{-\varepsilon}\left(\frac{a+b+2c-|M|+|N|}2+m,\frac{a+b+2c+|M|-|N|}2+m\right)}{H_q^{-\varepsilon}\left(\frac{a+b+2c-|M|-|N|}2+m,\frac{a+b+2c+|M|+|N|}2+m\right)}.
\end{align}
}
\end{theorem}

As an example, consider the case $(M,N)=(2,0)$. 
Assuming also that $n$ is odd, we have
$$Q^{2,0,n}(\alpha,\beta,\gamma;q)
=\left|\begin{matrix}Q_{11} &Q_{12}\\Q_{21} & Q_{22}\end{matrix}\right|, $$
where
{\allowdisplaybreaks
\begin{align*}
Q_{11}&=(\alpha-\alpha^{-1})\frac{(\alpha^2q^2,\alpha^2\gamma^2;q^2)_{(n-1)/2}}{q^{\frac 14(n-1)^2}\alpha^{n-1}\gamma^{\frac 12(n-1)}},\\
Q_{12}&=(\alpha q-\alpha^{-1}q^{-1})\frac{(\alpha^2q^2,\alpha^2\gamma^2;q^2)_{(n-1)/2}(1-\alpha^2\beta^2)}{q^{\frac 14(n-1)^2}\alpha^{n}\beta\gamma^{\frac 12(n-1)}}\\
&\quad\times
\left(1+\frac{(1-q^{1-n})(1-\alpha^2\beta^2\gamma^2q^{n-1})(1-\alpha^4q^2)(1-q^{-2})}{(1-q^2)(1-\alpha^2q^2)(1-\alpha^2\beta^2)(1-\alpha^2\gamma^2)}\,q^2\right)
,\\
Q_{21}&=\frac{(\alpha^2,\alpha^2\gamma^2;q^2)_{(n+1)/2}}{q^{\frac 14(n+1)(n-1)}\alpha^{n+1}\gamma^{\frac 12(n+1)}},\\
Q_{22}&=\frac{(\alpha^2,\alpha^2\gamma^2;q^2)_{(n+1)/2}(1-\alpha^2\beta^2)}{q^{\frac 14(n+1)(n-1)}\alpha^{n+2}\beta\gamma^{\frac 12(n+1)}}\\
&\quad\times
\left(1+\frac{(1-q^{-1-n})(1-\alpha^2\beta^2\gamma^2q^{n-1})(1-\alpha^4q^2)(1-q^{-2})}{(1-q^2)(1-\alpha^2)(1-\alpha^2\beta^2)(1-\alpha^2\gamma^2)}\,q^2\right)
.\\
\end{align*}
}
This can be simplified to
\begin{multline*}Q^{2,0,n}(\alpha,\beta,\gamma;q)=\frac{(1-q)(\alpha^2;q^2)_{(n+1)/2}(\alpha^2q^2,\alpha^2\gamma^2,\alpha^2\gamma^2q^2;q^2)_{(n-1)/2}}{q^{\binom n2+1}\alpha^{2n+2}\beta\gamma^n}\\
\times\big\{(1+\alpha^2q)(1-\alpha^2\beta^2)(1-\alpha^2\gamma^2)-(1+q^{-n})(1-\alpha^2\beta^2\gamma^2q^{n-1})(1-\alpha^4q^2)\big\}. \end{multline*}
If  $(\alpha,\beta,\gamma)=(q^{x/2},\pm q^{y/2},q^{z/2})$, the leading Taylor coefficient of this function at $q=1$ is a completely factored expression times
$$(x+y)(x+z)-2(x+1)(x+y+z+n-1)$$
Substituting $(x,y,z,n)\mapsto (-b-c-m-1,a+c+m+1,m+1,b)$ we find that, if $a$, $b$ and $c$ are all odd and $(M,N)=(2,0)$, then 
$Z(1)$ is a completely factored expression times
$$(b-a)(b+c)+2(b+c+m)(a+m)=(a+b)(b+c)+2m(a+b+c+m). $$
After interchanging $a$ and $b$, we recover the second half of
 \cite[Conj.\ 1]{cekz}.
In this way, \cite[Conj.\ 1 and Conj.\ 2]{cekz} can both be obtained as 
special cases of Theorem \ref{mt}.

An intriguing consequence of Theorem \ref{mt} is that $Z(q)$ is invariant under
the transformation $(M,N)\mapsto(-M,-N)$, up to an elementary prefactor. 
This means that the position of $T$ is reflected in the center of $H$, see Figure \ref{symmfig}.  It would  be interesting to have a conceptual explanation for this unexpected symmetry.

\begin{corollary}
Denoting by $Z_{MN}$ the partition function $Z(q)$ with fixed values of $a$, $b$, $c$, $m$ and $q$, we have
\begin{align*}\frac{Z_{MN}}{Z_{-M,-N}}&=\frac 1{2^{m(M+N)}}
\frac{H_q\left(\frac{a+b+M+N}2,\frac{a+b-M-N}2+m\right)}{H_q\left(\frac{a+b-M-N}2,\frac{a+b+M+N}2+m\right)}\\
&\quad\times\frac{H_{q^2}\left(\frac{a+c-N}2,\frac{a+c+N}2+m,\frac{b+c-M}2,\frac{b+c+M}2+m\right)}{H_{q^2}\left(\frac{a+c+N}2,\frac{a+c-N}2+m,\frac{b+c+M}2,\frac{b+c-M}2+m\right)}. \end{align*}
\end{corollary}

\begin{figure}
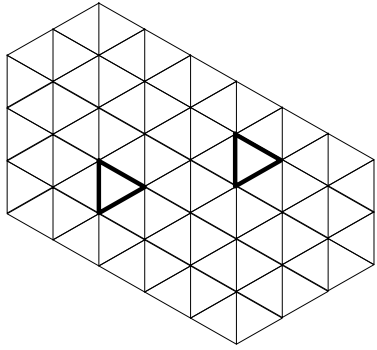
 
\centertexdraw{
\drawdim cm \setunitscale.7
\linewd.01
\rhombus \rhombus \rhombus \rhombus \rhombus \rhombus 
\ldreieck
 \move(-0.866025 -.5)
 \rhombus \rhombus \rhombus \rhombus \rhombus \rhombus \rhombus 
\ldreieck
 \move(-0.866025 -.5)
 \rdreieck \rhombus \rhombus \rhombus \rhombus \rhombus\rhombus\rhombus
\move(-0.866025 -1.5)
 \rdreieck \rhombus \rhombus \rhombus \rhombus \rhombus\rhombus
\move(-0.866025 -2.5)
 \rdreieck \rhombus \rhombus \rhombus \rhombus \rhombus
  \linewd.08
\move(0.866025 -2.5)\rdreieck
\move(3.464102 -2)\rdreieck
 }
\caption{Removing one of the two indicated triangles leads to partition functions related by an elementary multiplier.}
\label{symmfig}
\end{figure}

As an example, removing the left triangle in Figure \ref{symmfig}
corresponds to 
$$(a,b,c,m,M,N)=(2,5,2,1,1,2),$$
which gives
$$Z(q)=\frac{(1+q)^4(1+q^2)^5(1+q^3)^3(1+q^4)^5(1+q^5)}{2^{13}q^{28}}\,f(q), $$
with
$$f(q)=q^8+q^7+2q^6+3q^5+3q^4+3q^3+2q^2+q+1. $$
For the right triangle, corresponding to $(a,b,c,m,M,N)=(2,5,2,1,-1,-2)$, 
$$Z(q)=\frac{(1+q)^3(1+q^2)^4(1+q^3)^2(1+q^4)^4(1-q^{10})}{2^{10}q^{25}(1-q)}\,f(q). $$
In particular, substituting $q=1$ we find that there are $544=2^5\cdot 17$ tilings in the first case and $1360=2^4\cdot 5\cdot 17$ in the second case, where 
the symmetry is responsible for
 the relatively large common prime factor $f(1)=17$.

\section{Proof of Theorem \ref{mt}}

\subsection{Lattice paths}\label{gvs}

Following \cite{cekz}, we  study the partition function $Z(q)$ by applying
a bijection from lozenge tilings to families of non-intersecting  paths in the square lattice.
Given a tiling of $H\setminus T$, we mark the midpoints of the edges on 
the side  $b+m$ 
  and construct paths ending at these points by
  following the direction of the lozenges. 
This gives $m$ paths starting at the adjacent side of $T$ and
 $b$ paths starting at the side $b$.
 We then apply an affine transformation mapping
the steps of the paths to edges in the square lattice.

\begin{figure}
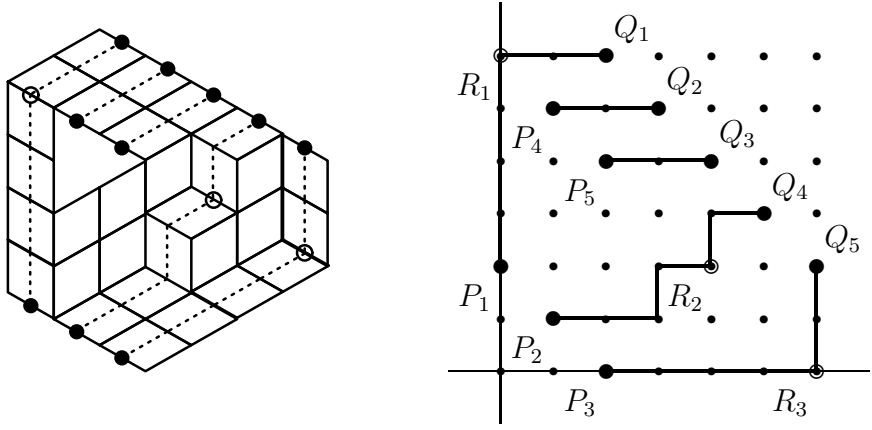
 
\hspace*{-4cm}
\begin{texdraw}
\drawdim cm \setunitscale.7
  \linewd.05
 \move(0 0)
 \RhombusA  \RhombusA \RhombusA \RhombusA \RhombusB \RhombusB \RhombusA
 \move(-.866025 -.5)
  \RhombusA \RhombusA \RhombusA \RhombusB \RhombusA \RhombusB 
  \RhombusA 
  \move(0.866025 -2.5)
\RhombusB \RhombusB \RhombusA \RhombusA
\move(0 -3)
\RhombusB \RhombusB \RhombusA \RhombusA
 \move(-.866025 -0.5) \RhombusC 
 \move(-.866025 -1.5) \RhombusC 
 \move(-.866025 -2.5) \RhombusC 
 \move(-.866025 -3.5) \RhombusC 
 \move(4.330127 -1.5) \RhombusC 
\move(4.330127 -2.5) \RhombusC 
\ringerl(1.29904 0.25)
\odSchritt \odSchritt \vdSchritt\vdSchritt\vdSchritt\vdSchritt
\ringerl(2.16506 -0.25)
\odSchritt\odSchritt
\ringerl(3.03109 -0.75)
\odSchritt\odSchritt
\ringerl(3.89711 -1.25)
\odSchritt\vdSchritt\odSchritt\vdSchritt\odSchritt\odSchritt
\ringerl(4.76314 -1.75)
\vdSchritt\vdSchritt\odSchritt\odSchritt\odSchritt\odSchritt
\ringerl(0.43301 -1.25)
\ringerl(1.29904 -1.75)
\ringerl(-0.43301 -4.75)
\ringerl(0.43301 -5.25)
\ringerl(1.29904 -5.75)
\move(-0.43301 -.75)\lcir r:.15
\move(3.03109 -2.75)\lcir r:.15
\move(4.76314 -3.75)\lcir r:.15
\end{texdraw}
\hspace{2cm}
$
\Einheit=.7cm
\Gitter(7,7)(0,0)
\Koordinatenachsen(7,7)(0,0)
\Pfad(0,2),222211\endPfad
\Pfad(1,1),112121\endPfad
\Pfad(2,0),111122\endPfad
\Pfad(1,5),11\endPfad
\Pfad(2,4),11\endPfad
 \DickerPunkt(0,2) \DickerPunkt(1,1) \DickerPunkt(2,0) \DickerPunkt(1,5) \DickerPunkt(2,4) 
 \DickerPunkt(2,6) \DickerPunkt(3,5) \DickerPunkt(4,4) \DickerPunkt(5,3) \DickerPunkt(6,2)
\Kreisbig(0,6)\Kreisbig(4,2)\Kreisbig(6,0)
 \Label\lu{P_1}(0,2) 
 \Label\lu{P_2}(1,1)
 \Label\lu{P_3}(2,0)
\Label\lu{P_4}(1,5)\Label\lu{P_5}(2,4)
\Label\lu{R_1}(0,6)\Label\lu{R_2}(4,2)\Label\lu{R_3}(6,0)
\Label\ro{Q_1}(2,6)\Label\ro{Q_2}(3,5)\Label\ro{Q_3}(4,4)\Label\ro{Q_4}(5,3)\Label\ro{Q_5}(6,2)
$
\caption{Lattice paths corresponding to the tiling in Figure \ref{tiling}.
The circled points, denoted $R_j$,  will appear in \S \ref{os}.}
\label{lpfig}
\end{figure}

More precisely, with the conventions illustrated in Figure \ref{lpfig}, 
tilings of $H\setminus T$ are in bijection with families of up-right non-interesecting paths starting at the points $(P_j)_{j=1}^{b+m}$ and ending at the points  $(Q_j)_{j=1}^{b+m}$, where 
\begin{subequations}\label{sep}
\begin{align}
P_j&=\begin{cases} (j-1,b-j), & 1\leq j\leq b,\\
(C-b+j-1,A+b+m-j), & b+1\leq j\leq b+m,\end{cases}\\
Q_j&=(a+j-1,b+c+m-j), \qquad 1\leq j\leq b+m.
\end{align}
\end{subequations}
In this setting, the weight function becomes a weight on horizontal steps, 
given by 
\begin{equation}\label{sw}\frac{q^{(x+2y-Z)/2}+q^{(-x-2y+Z)/2}}{2} \end{equation}
for the step from $(x-1,y)$ to $(x,y)$, where
$$Z=2A+C+m-1. $$

We now recall the Gessel--Viennot method for weighted enumeration of lattice paths \cite{gv}. 
Consider,  in general,  an arbitrary weight 
 assigned to each horizontal edge in 
the square lattice. 
We define the weight of a family of paths
to be the product of the weights of all horizontal steps in the family.
Let
$$w(P_1,\dots,P_n;Q_1,\dots,Q_n)$$ 
denote the sum of the weight of all 
families of $n$ non-intersecting up-right lattice paths, where 
the $i$:th path starts at
 $P_i$ and ends at $Q_i$, for $1\leq i\leq n$.
We then have the following fundamental result of Lindstr\"om \cite{l}, see 
\cite{kr2} for a historical discussion.

\begin{lemma}[Lindstr\"om]\label{ll}
The following identity holds:
\begin{equation}\label{li}\det_{1\leq j,k\leq n}\big(w(P_j;Q_k)\big)=\sum_{\sigma\in S_n}\sgn(\sigma)w(P_1,\dots,P_n;Q_{\sigma(1)},\dots,Q_{\sigma(n)}).\end{equation}
\end{lemma}

In other words, in the Laplace expansion of the left-hand side, the contributions from intersecting lattice paths cancel.
Consider now the case when the points are given by \eqref{sep}. 
Since the paths starting at $T$ will end at consecutive points
on the side $b+m$, not all permutations will contribute to the sum in
\eqref{li}. In particular, if $m$ is even, only even permutations
contribute, which means that the right-hand side of \eqref{li} reduces to
our partition function.

\begin{corollary}\label{pdc} If the side length $m$ of the triangle $T$ is even, the partition function \eqref{z} has the determinant representation
\begin{equation}\label{pdi}Z(q)=\det_{1\leq j,k\leq b+m}\big(w(P_j;Q_k)\big). 
\end{equation}
\end{corollary}

If $m$ is odd, we get instead a determinant representation for 
a sign-variation of the
partition function, generalizing
 the $(-1)$-enumeration studied in \cite{cekz}. Although
our methods can be adapted to study this function, it will not be
considered in the present work.

\subsection{The partition function as a discrete Selberg integral}
\label{dss}

Continuing in the footsteps of Ciucu et al.\ \cite{cekz}, 
we rewrite the determinant of Corollary~\ref{pdc} 
as a discrete Selberg integral.
The following result is
 a  special case of \cite[Prop.~2.1~(c)]{sc2}.

\begin{lemma}[Schlosser]\label{sdl}
One has the determinant evaluation
\begin{align*}\nonumber&\det_{1\leq j,k\leq m}\left(w\big((x_1+j-1,y_1+m-j);(x_2+l_k,y_2-l_k)\big)\right)\\
&\nonumber\quad={(-1)^{\binom m2}2^{\binom m2-m(x_2-x_1)-|l|}}q^{X}\Delta(q^l)\\
&\quad\quad\times\prod_{j=1}^m\frac{(q;q)_{x_2+y_2-x_1-y_1-m+j}(-q^{Z-x_2-y_1-y_2-m+j};q)_{x_2-x_1-j+1+l_j}}{(q;q)_{x_2-x_1+l_j}(q;q)_{y_2-y_1-l_j}}, \end{align*}
where $|l|=\sum_{j=1}^m l_j$
and
\begin{align*}X&=-\frac 32\binom m3+\frac 12\binom m2(- 3x_1+2x_2-2y_1+Z-1)\\
&\quad+\frac 14\,m(x_2-x_1)(x_1+x_2+4y_1-2Z+1)\\
& \quad+\frac 12\sum_{j=1}^m\binom{l_j}2+\frac 12\,|l|(x_2+2y_1-Z+1).\end{align*}
\end{lemma}

The exponent $X$ looks complicated, but can be determined from the fact that
the weight is invariant under $q^{1/2}\mapsto q^{-1/2}$.
If we would use a more symmetric notation (based on $q$-numbers $q^{-a/2}-q^{a/2}$ rather than $1-q^a$), the resulting expression would be simpler.

Let us now expand
the determinant in Corollary \ref{pdc} into minors according to
\begin{align*}\det_{1\leq j,k\leq b+m}\big(w(P_j;Q_k)\big)
&=\sum_{0\leq l_1<\dots<l_m\leq b+m-1}(-1)^{|l|+\binom m2+bm}\\
&\quad\times\det_{1\leq j,k\leq b}\big(w(P_j;Q_{\hat l_k+1})\big)
\det_{1\leq j,k\leq m}\big(w(P_{b+j};Q_{l_k+1})\big),\end{align*}
where $\hat l_1<\dots<\hat l_b$ is the ordered complement of $\{l_1,\dots,l_m\}$
in $[0,b+m-1]$. The determinants on the right-hand side may be evaluated using Lemma \ref{sdl}.
 Rewriting all factors involving the indices $\hat l_j$ in terms of $l_j$,
using in particular
$$\Delta(q^{\hat l})=(-1)^{\binom b2+\binom m2}q^{\binom{b+m}3+(2-b-m)|l|+\sum_{j=1}^m\binom {l_j}2}
\frac{\prod_{j=1}^{b+m}(q;q)_{j-1}}{\prod_{j=1}^m(q;q)_{l_j}(q;q)_{b+m-1-l_j}}\,\Delta(q^{l})
 $$
 gives
\begin{multline*}\det_{1\leq j,k\leq b+m}\big(w(P_j;Q_k)\big)
=\frac {q^X}{2^{m(a+b-C)+ab}}\prod_{j=1}^{b+m}\frac{(q;q)_{j-1}(-q^{A-B-b+1};q)_{a+j-1}}{(q;q)_{a+j-1}(q;q)_{c+j-1}}\\
\times\prod_{j=1}^b\frac{(q;q)_{a+c+m+j-1}}{(-q^{A-B-b+1};q)_{j-1}}
\sum(-1)^{|l|+bm+\binom m2}q^{\sum_{j=1}^m\binom{l_j}2+|l|(A-b-m+2)}\Delta(q^l)^2\\
\times\prod_{j=1}^m\frac{(q;q)_{B+j-1}(q;q)_{a+l_j}(q;q)_{b+c+m-1-l_j}(-q^{-B-m+j};q)_{a-C-j+1+l_j}}{(q;q)_{l_j}(q;q)_{a-C+l_j}(q;q)_{b+m-1-l_j}(q;q)_{b+c+m-A-1-l_j}(-q^{A-B-b+1};q)_{a+l_j}},
\end{multline*}
where the sum is over indices satisfying
\begin{equation}\label{mi}\max(0,C-a)\leq l_1<\dots<l_m\leq \min(b+m-1,b+c+m-A-1)\end{equation}
and where
\begin{align*}
X&=\frac 12(B-A-c)\binom m2+\frac 14 \big((a-C)(a-C+1)+b(3b-4A-2C-1)\big)m\\
&\quad
+\frac 14\,ab(a+3b-4A-2C).
\end{align*}
In the notation
 \eqref{dc} and \eqref{mhs}, this identity can be expressed as follows, 
where we have rewritten the prefactor in a way that will be convenient later.

\begin{proposition}\label{dsp}
\begin{flalign}
\nonumber&\shoveleft{\det_{1\leq j,k\leq b+m}\big(w(P_j;Q_k)\big)}\\
\nonumber&
=\frac {q^X}{2^{\frac 12{m(a+b+M+N)}+ab}}
\prod_{j=1}^b\frac{(q;q)_{a+c+m+j-1}(q;q)_{j-1}(-q^{\frac 12(-a-b+M-N)+j};q)_{a}}{(q;q)_{a+j-1}(q;q)_{c+m+j-1}}\\
\nonumber&\quad\times
\prod_{j=1}^m\left(\frac{(q^{1-b-c-m},q^{1+\frac 12(a-b+M+N)},-q^{1+\frac 12(a-b+M-N)};q)_{b+j-1}}{(q^{1-b-m};q)_{j-1}(q^{2-2m-b-c+M};q^2)_{j-1}(q^{a+1};q)_{b+j-1}}\right.\\
\nonumber&\qquad\qquad\qquad\qquad\qquad\qquad{\times\left.\frac{(q^2;q^2)_{\frac 12(a+c+N)+j-1}}{(q;q)_{\frac{a+b+M+N}2+j-1}(q^2;q^2)_{\frac{b+c-M}2+j-1}}\right)}\\
&\quad\times {}_{4}\phi_3^{(m)}\left(\begin{matrix}q^{1-b-m},q^{a+1},q^{1-\frac b2-\frac c2+\frac M2-m},-q^{1-\frac b2-\frac c2+\frac M2-m}\\q^{1-b-c-m},q^{1+\frac a2-\frac b2+\frac M2+\frac N2},-q^{1+\frac a2-\frac b2+\frac M2-\frac N2}\end{matrix};q,q\right),
\label{wdds}\end{flalign}
where
\begin{align*}
X&=-2\binom m3-\frac 14(a-3b+M+N+4)\binom m2-\frac 14\,ab(2c+M-N)\\
&\quad-\left(\frac 12\,b(a-b)+\frac 1{16}(a-3b+M+N)(a+b+4c-3M+N+2)
\right)m,
\end{align*}
\end{proposition}

Note that the factor $\prod_{j=1}^m{(q^{1+\frac 12(a-b+M+N)};q)_{b+j-1}}$
may vanish. If this is the case, the ${}_4\phi_3^{(m)}$
  in \eqref{wdds}  should be interpreted as a sum over indices
$l_j$ such that this zero is cancelled by 
$\prod_{j=1}^m{(q^{1+\frac 12(a-b+M+N)};q)_{l_j}^{-1}}$. This  
 gives the restriction 
 $l_1\geq(-a+b-M-N)/2=C-a$ as in \eqref{mi}.

\subsection{The partition function as a function of $m$}
\label{os}

Corollary \ref{pdc} is only valid for even values of $m$.
In order to study the case of odd $m$, we will need the following fact.
An analogous result for the enumeration $Z(1)$
was proved in \cite[\S 6]{cekz}. 
The simple proof given there seems difficult to extend to the weighted 
enumeration,  so we  use a slightly different approach.

\begin{lemma}\label{pl} For fixed values of $a$, $b$, $c$, $A$, $B$ and $C$,
\begin{equation}\label{pli}Z(q)=\prod_{j=1}^m\left(\frac{q^{\frac j2}+q^{-\frac j2}}2\right)^{a+b-C}f\left(q^{\frac m2}\right),\end{equation}
where $f$ is a rational function independent of $m$.
\end{lemma}

\begin{proof} Given a lozenge tiling of $H\setminus T$, we split
 $H$ in two parts, separated by the line $L$ containing the side
of $T$ parallel to the side $b$. Recall 
 the bijection to lattice paths 
described in \S \ref{gvs}. The paths starting at the side $b$  cross
$L$ at $b$ line-segments, which are marked with circles  in the left part of Figure \ref{lpfig}.
 Ordering them from left to right, let $x_j$ denote the distance from the $j$:th segment to $T$. If $l$ is the number of crossings to the left of $T$, the numbers 
$x_j$ are restricted by
\begin{subequations}\label{cdr}
\begin{align}
0&\leq x_l<\dots<x_1\leq\min(C-1,b+c-A-1),\\
0&\leq x_{l+1}<\dots<x_b\leq\min(A-1,a+b-C-1).
\end{align}
\end{subequations}

Applying the same affine map as in \S \ref{gvs}, the circled points
 are mapped to
$$R_j=\begin{cases}(C-x_j-1,A+x_j+m), & 1\leq j\leq l,\\
(C+x_j+m,A-x_j-1), & l+1\leq j\leq b.\end{cases} $$
The lattice paths split into a family of $b$ paths starting at $(P_j)_{j=1}^b$ and ending at $(R_j)_{j=1}^b$ and a family of
$b+m$ paths starting at $(S_j)_{j=1}^{b+m}$ and ending at  $(Q_j)_{j=1}^{b+m}$,
where
$$(S_1,\dots,S_{b+m})=(R_1,\dots,R_l,P_{b+1},\dots,P_{b+m},R_{l+1},\dots,R_b). $$
More explicitly, $S_j=(C+y_j,A+m-y_j-1)$, where
$$(y_1,\dots,y_{b+m})=(-x_1-1,\dots,-x_l-1,0,1,\dots,m-1,x_{l+1}+m,\dots,x_b+m). $$
Applying Lemma \ref{ll} to both families, only the identity permutation 
contributes to \eqref{li}, and we find that
\begin{equation}\label{zde}Z(q)=\sum\det_{1\leq j,k\leq b}\big(w(P_j;R_k)\big)\det_{1\leq j,k\leq b+m}\big(w(S_j;Q_k)\big), \end{equation}
where the sum is over all solutions to \eqref{cdr}, for $0\leq l\leq b$.
In contrast to \eqref{pdi}, \eqref{zde}  holds regardless of the
parity of $m$.

The first determinant in \eqref{zde} can be computed using Lemma \ref{sdl}.
By the symmetry
$$w_Z((a,b);(c,d))=w_{1-Z}((-c,-d);(-a,-b)), $$
where we indicate the $Z$-dependence in \eqref{sw} 
by a subscript, the second determinant can be expressed as
\begin{equation}\label{zsd}(-1)^{\binom{b+m}2}\det_{1\leq j,k\leq b+m}\big(w_{1-Z}(-a-b-m+j,1-c-j;-C-y_k,1-m-A+y_k)\big), \end{equation}
which is again computed by  Lemma \ref{sdl}.

It will be convenient to write $f\sim g$ if, as a function of $m$, 
$f/g=h(q^{m/2})$, with $h$  rational. 
We need to prove that 
\begin{equation}\label{zw}Z(q)\sim W^{a+b-C},\end{equation}
 where
$$W=\prod_{j=1}^m\frac{q^{\frac j2}+q^{-\frac j2}}2={2^{-m}q^{-\frac 12\binom{m+1}2}(-q;q)_m}.$$
Consider first
\eqref{zsd}, which is obtained by substituting 
\begin{multline*}(m,x_1,x_2,y_1,y_2,l_k,Z)\\
\mapsto(b+m,1-a-b-m,-C,1-b-c-m,1-m-A,-y_k,2-2A-C-m) \end{multline*}
in
 Lemma \ref{sdl}. Under this substitution,
$2^{\binom m2-m(x_2-x_1)-|l|}\mapsto 2^{m(C-a-l)}$ and
$$q^X\sim q^{2\binom m3+\frac{4b-a+C-3l+2}2\binom {m+1}2}.$$ 
The factor $\Delta(q^l)\mapsto\prod_{1\leq j<k\leq b+m}(q^{-y_k}-q^{-y_j})$ splits  into six parts, depending on whether $j$ and $k$ belong to the interval $[1,l]$,  $[l+1,l+m]$ or $[l+m+1,b+m]$. 
The three parts with neither $j$ nor $k$ in the middle interval 
are clearly rational in $q^m$. This leaves us with 
$$\prod_{1\leq j\leq l,\,1\leq k\leq m}(q^{1-k}-q^{x_j+1})\prod_{1\leq j<k\leq m}(q^{1-k}-q^{1-j})
\prod_{1\leq j\leq m,\,1\leq k\leq b-l}(q^{-m-x_k}-q^{1-j}), $$
where the first factor can be written
$$q^{-l\binom m2}\prod_{j=1}^l(q^{x_j+1};q)_m=q^{-l\binom m2}(q;q)_m^l\prod_{j=1}^l\frac{(q^{m+1};q)_{x_j}}{(q;q)_{x_j}}\sim q^{-l\binom {m+1}2}(q;q)_m^l. $$
The second factor is equal to $q^{-2\binom m3-\binom{m+1}2+m}\tilde H_q(m)$ and the third factor equivalent to $q^{-2(b-l)\binom {m+1}2}(q;q)_m^{b-l}$. Next, we have
$$\prod_{j=1}^m(q;q)_{x_2+y_2-x_1-y_1-m+j}\mapsto\prod_{j=1}^{b+m}(q;q)_{B+j-1}\sim \tilde H_q(b+B+m), $$
\begin{align*}\prod_{j=1}^m(q;q)_{x_2-x_1+l_j}&\mapsto\prod_{j=1}^l(q;q)_{a+b-C+m+x_j}\prod_{j=1}^m(q,q)_{a+b-C+j-1}\prod_{j=l+1}^{b}(q;q)_{a+b-C-x_j-1}\\
&\sim(q;q)_m^l\tilde H_q(a+b-C+m),\\
\prod_{j=1}^m(q;q)_{y_2-y_1-l_j}&\mapsto\prod_{j=1}^l(q;q)_{b+c-A-x_j-1}\prod_{j=1}^m(q,q)_{b+c-A+j-1}\prod_{j=l+1}^{b}(q;q)_{b+c-A+x_j+m}\\
&\sim(q;q)_m^{b-l}\tilde H_q(b+c-A+m).
 \end{align*} 
By Lemma \ref{hrl}, 
$$\frac{\tilde H_q(m)\tilde H_q(b+B+m)}{\tilde H_q(a+b-C+m)\tilde H_q(b+c-A+m)}\sim 1. $$
Finally, we have
\begin{multline*}\prod_{j=1}^m(-q^{Z-x_2-y_1-y_2-m+j};q)_{x_2-x_1-j+1+l_j}
\mapsto\prod_{j=1}^l(-q^{c-A+j};q)_{a+b-C+x_j+1+m-j}\\
\times\prod_{j=1}^m(-q^{c-A+l+j};q)_{a+b-C+m-l+1-2j}\prod_{j=l+1}^{b}(-q^{c-A+m+j};q)_{a+b-C-m-x_j-j}, \end{multline*}
where the first factor is equivalent to $(-q;q)_m^l$ 
and the third factor to $(-q;q)_m^{l-b}$. If $a+b\geq C+l$, the  second factor can be expressed as
$$\prod_{j=1}^m(-q^{c-A+l+j};q)_{a+b-C-l}=\prod_{j=1}^{a+b-C-l}(-q^{c-A+l+j};q)_{m}\sim(-q;q)_m^{a+b-C-l}. $$
By a similar computation,  this holds also for $a+b<C+l$.
In conclusion, \eqref{zsd} is equivalent to
$W^{a+l-C}. $
Similarly, though with less effort, we find that the
first determinant in \eqref{zde} is equivalent to
$W^{b-l}$, which gives \eqref{zw}.
\end{proof}

\subsection{Discrete Selberg integrals and Askey--Wilson polynomials}
\label{aws}

We  recall some basic facts on Askey--Wilson polynomials \cite{aw}. Normalizing them to be monic (which is not the standard choice in the literature), they are 
given by
\begin{align}\notag P_n(x)&=P_n\left(x;a,b,c,d;q\right)\\
\label{awp}&=\frac{(ab,ac,ad;q)_n}{2^na^n(abcdq^{n-1};q)_n}\,{}_4\phi_3\left(\begin{matrix}q^{-n},abcdq^{n-1},a\xi,a/\xi\\ab,ac,ad\end{matrix};q,q\right),\end{align}
where $x=(\xi+\xi^{-1})/2$.
When $\max(|a|,|b|,|c|,|d|,|q|)<1$, 
they satisfy the orthogonality relation
$$\int_{-1}^1P_m(x)P_n(x)\,w(x)\,dx=C_n{\delta_{mn}}, $$
where, using standard notation such as  $(\xi^{\pm 2};q)_\infty=(\xi^{ 2};q)_\infty(\xi^{- 2};q)_\infty$,
\begin{subequations}\label{awo}
\begin{equation}\label{awm}w(x)=w(x;a,b,c,d;q)
=\frac{(\xi^{\pm 2};q)_\infty}{(a\xi^\pm,b\xi^\pm,c\xi^\pm,d\xi^\pm;q)_\infty\sqrt{1-x^2}}
 \end{equation}
and
\begin{align}C_n&=C_n(a,b,c,d;q)\nonumber \\
&=\frac{2\pi(abcdq^{2n-1},abcdq^{2n};q)_\infty}{4^n(q^{n+1},abq^n,acq^n,adq^n,bcq^n,bdq^n,cdq^n,abcdq^{n-1};q)_\infty}. \end{align}
\end{subequations}
The polynomial $P_n$ is symmetric in the parameters $a,b,c,d$.

To link Askey--Wilson polynomials to discrete Selberg integrals, we
will need the 
Cauchy--Binet identity
\begin{equation}\label{cbi}\det_{1\leq j,k\leq m}\left(\sum_{l=0}^NA_{jl}B_{lk}\right)
=\sum_{0\leq l_1<\dots<l_m\leq N}\det_{1\leq j,k\leq m}(A_{j,l_k})\det_{1\leq j,k\leq m}(B_{l_k,j}) \end{equation}
as well as the determinant evaluation 
\begin{equation}\label{ksd}\det_{1\leq j,k\leq m}\big((aq^{j-1},bq^{m-j};q)_{l_k}\big)
=q^{\binom m3}b^{\binom m 2}\prod_{j=1}^m\frac{(a,b;q)_{l_j}(q^{j-m}a/b;q)_{j-1}}{(a,b;q)_{j-1}}\,\Delta(q^l).\end{equation}
To prove the latter, we write the determinant as
\begin{multline*}\det_{1\leq j,k\leq m}\left(\frac{(a,b;q)_{l_k}}{(a;q)_{j-1}(b;q)_{m-j}}(aq^{l_k};q)_{j-1}(bq^{l_k};q)_{m-j}\right)\\
=\prod_{j=1}^m\frac{(a,b;q)_{l_j}}{(a,b;q)_{j-1}}\det_{1\leq j,k\leq m}\big((aq^{l_k};q)_{j-1}(bq^{l_k};q)_{m-j}\big), \end{multline*}
which can be evaluated using  \cite[Lemma 2.2]{kr} or \cite[Lemma~A.1]{sc1}.

\begin{proposition}\label{dsl}
Let
$$x_k=\frac{\xi q^{k-1}+\xi^{-1}q^{1-k}}{2},\qquad k=1,\dots,m $$
and let $P_n(x)$ be the monic Askey--Wilson polynomial
\eqref{awp}. Then,
\begin{multline}\label{adds}
\frac{\det_{1\leq j,k\leq m}\big(P_{n+j-1}(x_k)\big)}{\Delta(x)}=
\frac{1}{q^{2\binom m3+(n+1)\binom m2}(2a)^{mn}}\\
\begin{split}&\times\prod_{j=1}^m\frac{(ab,ac,ad;q)_{n+j-1}}{(abcdq^{n-1};q)_{n+j-1}(q,q^{1-m-n},a\xi,a\xi^{-1}q^{1-m};q)_{j-1}}\\
&\times
{}_{4}\phi_3^{(m)}\left(\begin{matrix}q^{1-m-n},abcdq^{n-1},a\xi,aq^{1-m}\xi^{-1}\\
ab,ac,ad\end{matrix};q,q\right).
\end{split}\end{multline}
\end{proposition}

\begin{proof}
Since
\begin{align*}P_{n+j-1}(x_k)&=\frac{(ab,ac,ad;q)_{n+j-1}}{(2a)^{n+j-1}(abcdq^{n+j-2};q)_{n+j-1}}\\
&\quad\times\sum_{l\geq 0}\frac{(q^{1-j-n},abcdq^{n+j-2},a\xi q^{k-1},a\xi^{-1}q^{1-k};q)_l }{(q,ab,ac,ad;q)_l}\,q^l, \end{align*}
expanding the left-hand side of \eqref{adds} using \eqref{cbi} gives
\begin{multline*}\prod_{j=1}^m\frac{(ab,ac,ad;q)_{n+j-1}}{(2a)^{n+j-1}(abcdq^{n+j-2};q)_{n+j-1}}
\sum_{0\leq l_1<\dots<l_m\leq n+m-1}\prod_{j=1}^m\frac{q^{l_j}}{(q,ab,ac,ad;q)_{l_j}}\\
\times\det_{1\leq j,k\leq m}\big((q^{1-j-n},abcdq^{n+j-2};q)_{l_k}\big)
\det_{1\leq j,k\leq m}\big((a\xi q^{j-1},a\xi^{-1}q^{1-j};q)_{l_k}\big).
 \end{multline*}
Applying \eqref{ksd} and simplifying, using also
$$\Delta(x)=\frac{\prod_{j=1}^m(q,q^{j-1}\xi^2;q)_{j-1}}{q^{2\binom m3+\binom m2}(2\xi)^{\binom m2}}, $$
completes the proof.
\end{proof}

We note that, since the Askey--Wilson polynomial is 
symmetric in its parameters, the right-hand side of \eqref{adds} is invariant under interchanging $a$ and $b$. This proves the following multiple Sears' transformation, which is
a very special case of a transformation for discrete elliptic Selberg integrals conjectured by
Warnaar \cite{w} and proved by Rains \cite{ra}. We will use this transformation
 in \S \ref{css}.

\begin{corollary}[Rains]\label{msc}
If $q^{1-n}abc=def$, then
\begin{align*}
{}_{4}\phi_3^{(m)}\left(\begin{matrix}q^{1-m-n},a,b,c\\
dq^{m-1},e,f\end{matrix};q,q\right)&=\left(\frac{bc}d\right)^{mn} \prod_{j=1}^m\frac{(b,c;a)_{j-1}(de/bc,df/bc;q)_{n+j-1}}{(d/b,d/c;q)_{j-1}(e,f;q)_{n+j-1}}\\
&\quad\times\,{}_{4}\phi_3^{(m)}\left(\begin{matrix}q^{1-m-n},a,d/b,d/c\\
dq^{m-1},de/bc,df/bc\end{matrix};q,q\right).
\end{align*}
\end{corollary}

\subsection{Continuous Selberg integrals}
\label{css}

Let $\mu$ be a linear functional on  $\mathbb C[x]$, which we write as
a formal integral
$$\mu(p)=\int p(x)\,d\mu(x). $$
We will assume that $\mu $ is non-degenerate in the sence that there exist
monic polynomials $p_n$ of degree $n$ such that
\begin{equation}\label{po}\mu(p_mp_n)=C_n\delta_{mn}. \end{equation}
We   do not require any positivity condition for $\mu$.
We then have the identity
\begin{multline}\label{ch}\int \Delta(x)^2\prod_{j=1}^n\prod_{k=1}^m(y_k-x_j)\,d\mu(x_1)\dotsm d\mu(x_n)\\
=n!\,C_0\dotsm C_{n-1}\frac{\det_{1\leq j,k\leq m}\big(p_{n+j-1}(y_k)\big)}{\Delta(y)},
\end{multline}
relating a Selberg-type integral to a determinant of orthogonal polynomials.
This identity
can be obtained by combining two classical results \cite[Thm.\ 2.1.2 and Thm.~2.7.1]{i} due to Heine and Christoffel. More explicitly, it appears in \cite{bh}. A direct proof is very easy; simply write the integrand as 
$$\frac{\Delta(x)\Delta(x,y)}{\Delta(y)}, $$ 
expand both factors in the numerator using $\Delta(x)=\det(p_{j-1}(x_i))$
and then integrate using \eqref{po}. We can apply \eqref{ch} to prove 
 the following quadratic transformation formula
for determinants of Askey--Wilson polynomials.

\begin{theorem}\label{qdtt}
Let $m$, $n$, $M$  and $N$ be non-negative integers, with $m$ even, and
let  $a$, $b$ and $q$ be generic parameters.
Let $p_n$ and $q_n$  denote the monic polynomials
\begin{align*}p_n(x)&=P_{n}(x;aq^{M},-a,bq^{N},-b;q),\\
q_n(x)&=\begin{cases}2^{-n/2}P_{n/2}(2x^2-1;-1,-q^{m+1},a^2,b^2;q^2),& n  \text{\emph{ even}},\\
2^{-(n-1)/2}x P_{(n-1)/2}(2x^2-1;-q^2,-q^{m+1},a^2,b^2;q^2),& n\text{\emph{ odd}}\end{cases}\\
\end{align*}
and let
\begin{align*}y_k&=\frac{\eta_k+\eta_k^{-1}}{2},& \eta_k&=\ti q^{k-\frac{m+1}2},\qquad k=1,\dots,m,\\
z_k&=\frac{\zeta_k+\zeta_k^{-1}}{2},& \zeta_k&=\begin{cases}aq^{k-1}, &k=1,\dots,M,\\
bq^{k-M-1}, &k=M+1,\dots,M+N.
\end{cases} \end{align*}
Then, 
\begin{equation}\label{qdt}\frac{\det_{1\leq j,k\leq m}\left(p_{n+j-1}(y_k)\right)}
{\Delta(y)}
=
C\frac{\det_{1\leq j,k\leq M+N}\left(q_{n+j-1}(z_k)\right)}{\Delta(z)}, 
 \end{equation}
where 
\begin{align*}C&=
\left(2^{M+N-m}q^{\binom M2+\binom N2-\frac{m^2}4}a^{M}b^{N}\right)^n\\
&\quad\times\prod_{j=1}^{n}\left(\frac{
(q^{2[j/2]+1},-a^2q^{2[(j-1)/2]+1},-b^2q^{2[(j-1)/2]+1},a^2b^2q^{2[j/2]-1};q^2)_{\frac m2}
}{(a^2b^2q^{2j-3},a^2b^2q^{2j-1};q^2)_{\frac m2}}\right.\\
&\quad\times\left.\frac{(a^2b^2q^{2j-3},a^2b^2q^{2j-2};q)_{M+N}}{(abq^{j-1},a^2b^2q^{j-2};q)_{M+N}(-a^2q^{j-1},-abq^{j-1};q)_M(-b^2q^{j-1},-abq^{j-1};q)_N}\right)
.
\end{align*}
\end{theorem}

\begin{proof}
Let $L$ denote the left-hand side of \eqref{qdt}.
 Since \eqref{qdt} is a rational identity, we may assume that
$\max(|a|,|b|,|q|)<1$.
We  then apply \eqref{ch} to write
\begin{equation}\label{di}L=\frac{1}{n!\prod_{k=0}^{n-1}C_k}\int_{[-1,1]^n} \Delta(x)^2\prod_{j=1}^n\prod_{k=1}^m(y_k-x_j)\prod_{j=1}^n\,w(x_j)\,dx_j, \end{equation}
where $w$ and $C_k$ are obtained by substituting $(a,b,c,d)\mapsto(aq^{M},-a,bq^{N},-b)$  in \eqref{awo}. 

We will now rewrite \eqref{di} in such a way that the roles of $m$ and $M+N$ are interchanged.
In the orthogonality measure, we write
$$\frac{(\xi^{\pm 2};q)_\infty}{(aq^{M}\xi^\pm,-a\xi^\pm,bq^{N}\xi^\pm,-b\xi^\pm;q)_\infty}
=\frac{(a\xi^\pm;q)_M(b\xi^\pm;q)_N(\xi^{\pm 4};q^2)_\infty}{(-\xi^{\pm 2},-q\xi^{\pm 2},a^2\xi^{\pm 2},b^2\xi^{\pm 2};q^2)_\infty}
 $$
and observe that
$$(a\xi^\pm;q)_M(b\xi^\pm;q)_N=2^{M+N}q^{\binom M2+\binom N2}a^Mb^N \prod_{k=1}^{M+N}(z_k-x). $$
We also write
$$\prod_{k=1}^m\left(y_k-x\right)=\frac{(\ti q^{\frac{1-m} 2}\xi^{\pm};q)_m}{2^m\ti^m}=
2^{-m}q^{-\frac{m^2}{4}}\frac{(-q\xi^{\pm 2};q^2)_{\infty}}{(-q^{m+1}\xi^{\pm 2};q^2)_{\infty}}. $$
Combining these  facts, we find that 
\begin{equation}\label{dti}L=\frac{D}{n!\prod_{k=0}^{n-1}C_k}\int_{[-1,1]^n} \Delta(x)^2\prod_{j=1}^n\prod_{k=1}^{M+N}(z_k-x_j)\prod_{j=1}^n\tilde w(x_j)\,dx_j,  \end{equation}
where
$$\tilde w(x)=\frac{(\xi^{\pm 4};q^2)_\infty}{(-\xi^{\pm 2},-q^{m+1}\xi^{\pm 2},a^2\xi^{\pm 2},b^2\xi^{\pm 2};q^2)_\infty\sqrt{1-x^2}} $$
and 
$$D=\left(2^{M+N-m}q^{\binom M2+\binom N2-\frac{m^2}4}a^{M}b^{N}\right)^n.$$

We now apply \eqref{ch} to the integral \eqref{dti}. Let
$$\mu(p)=\int_{-1}^1p(x)\,\tilde w(x)\,dx. $$
 Then, 
\begin{align*}
\tilde w(x)\,dx&=\frac 1 2\, w(y;-1,-q^{m+1},a^2,b^2;q^2)\,dy,\\
x^2\tilde w(x)\,dx&=\frac 1 8\, w(y;-q^2,-q^{m+1},a^2,b^2;q^2)\,dy,
\end{align*}
where $y=2x^2-1$. It follows that 
the  polynomials $q_n(x)$ satisfy 
$\mu(q_mq_n)=\tilde C_n\delta_{mn} $, with
$$\tilde C_n=\begin{cases} 4^{-k} C_{k}(-1,-q^{m+1},a^2,b^2;q^2),& n=2k,\\
 4^{-k-1} C_k(-q^2,-q^{m+1},a^2,b^2;q^2),& n=2k+1.\end{cases} $$
Thus, \eqref{ch} gives 
$$L=D\prod_{j=1}^{n}\frac{\tilde C_{j-1}}{ C_{j-1}}\frac{\det_{1\leq j,k\leq M+N}(q_{n+j-1}(z_k))}{\Delta(z)}. $$
Simplifying the expression for $\tilde C_{j-1}/C_{j-1}$, we arrive at the 
desired result.
\end{proof}

We will now combine \eqref{adds} and
\eqref{qdt}. Let us substitute
$(a,b,c,d,\xi)\mapsto(\ti \alpha q^{M},-\ti \alpha,\ti \beta q^{N},-\ti \beta,\ti q^{(1-m)/2})$ in \eqref{adds},  where $M$ and $N$ are non-negative integers. 
Then, the left-hand side of \eqref{adds} equals
the left-hand side of \eqref{qdt}, under the substitutions
 $(a,b)\mapsto
(\ti \alpha,\ti \beta)$. We rewrite the matrix entries $q_{n+j-1}(z_k)$ in terms
of ${}_4\phi_3$ series, using \eqref{awp} with the distinguished parameter $a$
chosen as $-\alpha^2$ for $k\leq M$ and as $-\beta^2$ for $k\geq M+1$.
Then, the matrix entries in \eqref{qdt} can be identified with those in \eqref{qd}. More precisely,
$$q_{n+j-1}(z_k)=C_jD_k\,Q_{jk}^{MNn}(\alpha,\beta,q^{\frac{m+1}2};q), $$
where
\begin{align*}C_j&=
 \frac{(\ti/2)^{n+j-1}
 q^{
\frac {m+1}2\left[\frac{n+j-1}2\right]
+\left[\frac{(n+j-2)^2}4\right]
}
(\alpha^2\beta^2;q^2)_{[(n+j-1)/2]}}{(\alpha^2\beta^2q^{m+2[(n+j)/2]-1};q^2)_{[(n+j-1)/2]}},
\\
D_k&=\begin{cases}\displaystyle\frac{(\alpha\beta)^{k-1}q^{\binom{k-1}2}}{(\alpha^2\beta^2;q^2)_{k-1}}, & k\leq M,\\
\displaystyle\frac {(\alpha\beta)^{k-M-1}q^{\binom{k-M-1}2}}{(\alpha^2\beta^2;q^2)_{k-M-1}}, & k\geq M+1.\end{cases}\end{align*}
We compute
$$\Delta(z)=\frac{\prod_{k=1}^{M}(q,-q^{j-1}\alpha^{2};q)_{j-1}(q^{j-M}\beta/\alpha,-q^{j-1}\alpha\beta;q)_N\prod_{j=1}^{N}(q,-q^{j-1}\beta^{2};q)_{j-1}}{2^{\binom{M+N}2}\ti^{\binom{M+N}2}q^{2\binom{M}3+2\binom{N}3+\binom M2+(M+1)\binom N2}\alpha^{\binom M2}
\beta^{\binom{N}2+MN}}
$$
and simplify the factors involving $\alpha\beta$ using
{
\allowdisplaybreaks
\begin{multline*}\frac 1{
\prod_{j=1}^M{(-\alpha\beta q^{j-1};q)_N(\alpha^2\beta^2;q^2)_{j-1}}
\prod_{j=1}^N{(\alpha^2\beta^2;q^2)_{j-1}}}\\
\begin{split}&\quad\times\prod_{j=1}^{M+N}\frac{(\alpha^2\beta^2;q^2)_{[(n+j-1)/2]}}{(\alpha^2\beta^2q^{m+2[(n+j)/2]-1};q^2)_{[(n+j-1)/2]}}\prod_{j=1}^n\frac{(\alpha^2\beta^2q^{2j-3},\alpha^2\beta^2q^{2j-2};q)_{M+N}}{(\alpha^2\beta^2q^{2j-1},\alpha^2\beta^2q^{2j-3};q^2)_{\frac m2}}\\
&\quad\times\prod_{j=1}^n
\frac{
(\alpha^2\beta^2q^{2[j/2]-1};q^2)_{\frac m2}}{(-\alpha\beta q^{j-1},\alpha^2\beta^2q^{j-2};q)_{M+N}(\alpha\beta q^{j-1};q)_M(\alpha\beta q^{j-1};q)_N}\\
&=\frac{\prod_{j=1}^M(\alpha\beta q^{n+j-1};q)_N}{\prod_{j=1}^{M+N}(\alpha^2\beta^2q^{2n+j-2};q)_{j-1}}\\
&\quad\times\prod_{j=1}^{m/2}\frac{(\alpha^2\beta^2q^{2j-3};q^2)_{\left[\frac{M+N+n+1}2\right]}(\alpha^2\beta^2q^{2j-1};q^2)_{\left[\frac{M+N+n}2\right]}}{(\alpha^2\beta^2q^{2j-3};q^2)_{M+N+n}(\alpha^2\beta^2q^{2j-1};q^2)_{n}}.
\end{split}\end{multline*}
}
In the exponent of $q$, we use
\begin{multline*}\sum_{j=1}^{M+N}\left(\frac{m+1}2\left[\frac{n+j-1}2\right]+\left[\frac{(n+j-2)^2}4\right] \right)\\
=\frac 12\left(\binom{M+N+n}3-\binom n3\right)+\frac m2\left(\left[\frac{(M+N+n-1)^2}4\right]-\left[\frac{(n-1)^2}4\right]\right). 
\end{multline*}
This yields the following result.

\begin{corollary}\label{dsdc}
For $m$, $n$, $M$ and $N$ non-negative integers, with $m$ even,
{\allowdisplaybreaks
\begin{multline*}{}_{4}\phi_3^{(m)}\left(\begin{matrix}q^{1-m-n},\alpha^2\beta^2q^{M+N+n-1},\alpha q^{M-\frac m2+\frac 12},-\alpha q^{M-\frac m2+\frac 12}\\
\alpha^2q^M,\alpha\beta  q^M,-\alpha\beta q^{M+N}\end{matrix};q,q\right)\\
\begin{split}&={(-1)^{\binom{M+N}2+n\left(M+N+\frac m2\right)}\alpha^{(M+m)n+2\binom M2+\binom N2}\beta^{(M+n)N+\binom M2+2\binom N2}q^X}\\
&\quad \times
\prod_{j=1}^M\frac{(\alpha\beta q^{j+n-1};q)_N}{(q,-\alpha^{2}q^{j-1};q)_{j-1}(q^{j-M}\beta/\alpha;q)_N}\prod_{j=1}^{N}\frac 1{(q,-\beta^{2}q^{j-1};q)_{j-1}}\\
&\quad\times\prod_{j=1}^{m/2}\frac{(\alpha^2\beta^2q^{2j-3};q^2)_{[(M+N+n+1)/2]}(\alpha^2\beta^2q^{2j-1};q^2)_{[(M+N+n)/2]}}{(\alpha^2\beta^2q^{2j-3};q^2)_{M+N+n}(\alpha^2\beta^2q^{2j-1};q^2)_{n}}
\\
&\quad\times \prod_{j=1}^m\frac{(\alpha^2\beta^2q^{M+N+n-1};q)_{n+j-1}(q,q^{1-m-n};q)_{j-1}(\alpha^2q^{2M-m+1};q^2)_{j-1}}
{(\alpha^2q^M,\alpha\beta  q^M,-\alpha\beta q^{M+N};q)_{n+j-1}}\\
&\quad\times\prod_{j=1}^{n}\frac{
(q^{2[j/2]+1},\alpha^2q^{2[(j-1)/2]+1},\beta^2q^{2[(j-1)/2]+1};q^2)_{\frac m2}
}{(\alpha^2q^{j-1};q)_M(\beta^2 q^{j-1};q)_N}\prod_{j=1}^{M+N}\frac 1{( \alpha^2\beta^2q^{j+2n-2};q)_{j-1}}\\
&\quad\times Q^{MNn}(\alpha,\beta,q^{\frac {m+1}2};q),
\end{split}\end{multline*}
}
where
\begin{align*}
X&=2\binom m3+\binom m2+3\binom M3+3\binom N3+\binom M2+(M+1)\binom N2\\
&\quad+\left(\binom M2+\binom N2+m\left(M+\frac m4-\frac 12\right)\right)n+\frac 12\left(\binom{M+N+n}3-\binom n3\right)\\
&\quad+\frac m2\left(\left[\frac{(M+N+n-1)^2}4\right]-\left[\frac{(n-1)^2}4\right]\right)
\end{align*}
\end{corollary}

We  now reformulate Corollary \ref{dsdc} in a way that will be adapted to our purpose. We first apply the identity
\begin{multline*}\prod_{j=1}^{M}\frac 1{(-\alpha^{2}q^{j-1};q)_{j-1}}\prod_{j=1}^n\frac{
(\alpha^2q^{2[(j-1)/2]+1};q^2)_{\frac m2}
}{(\alpha^2q^{j-1};q)_M}\\
\begin{split}&=(-1)^{n\left(M+\frac m2\right)}\alpha^{mn-M(M-1+2n)}
 q^{\frac{nm^2}4+m\big[\frac{(n-1)^2}4\big]-3\binom M3-(n+1)\binom M2-M\binom n2}\\
&\quad\times\prod_{j=1}^{M}\frac 1{(-q^{j+1-2M}/\alpha^{2};q)_{j-1}}\prod_{j=1}^n\frac{
(q^{-2[(j-1)/2]+1-m}/\alpha^2;q^2)_{\frac m2}
}{(q^{j+1-M-n}/\alpha^2;q)_M}.\end{split} \end{multline*}
We then multiply the left-hand side by
$$q^{-\frac 12mnM}\prod_{j=1}^m\frac{(\alpha\beta q^M,-\alpha\beta q^{M+N};q)_{n+j-1}}
{(\alpha^2q^{2M-m+1};q^2)_{j-1}} $$
and  make the change of variables
$$(M,N,\alpha,\beta)\mapsto\left(|M|,|N|,\alpha q^{-\frac{|M|}2},-\sgn(MN)\beta q^{-\frac{|N|}2}\right), $$
where we no longer require $M$ and $N$ to be non-negative. Thus, for $M\geq 0$ we consider 
\begin{multline}\label{sph}q^{-\frac 12mnM}\prod_{j=1}^m\frac{(\alpha\beta q^\frac{M+N}2,-\alpha\beta q^{\frac{M-N}2};q)_{n+j-1}}
{(\alpha^2q^{M-m+1};q^2)_{j-1}}\\
\times\,{}_{4}\phi_3^{(m)}\left(\begin{matrix}q^{1-m-n},\alpha^2\beta^2q^{n-1},\alpha q^{\frac M2-\frac m2+\frac 12},-\alpha q^{\frac M2-\frac m2+\frac 12}\\
\alpha^2,\alpha\beta  q^{\frac M2+\frac N2},-\alpha\beta q^{\frac M2-\frac N2}\end{matrix};q,q\right) \end{multline}
and for $M<0$ the same quantity with $(M,\beta)\mapsto(-M,-\beta)$.
However, by Corollary \ref{msc},  \eqref{sph} 
is  invariant under this transformation. Thus, we may take the left-hand side as \eqref{sph} regardless of the sign of $M$.
This leads to the following result.

\begin{corollary}\label{fpc}
For $m$, $n$, $M$ and $N$ integers, with $m$ and $n$ non-negative and $m$ even,
and $\varepsilon=\sgn(MN)$,
{\allowdisplaybreaks
\begin{multline*}
{}_{4}\phi_3^{(m)}\left(\begin{matrix}q^{1-m-n},\alpha^2\beta^2q^{n-1},\alpha q^{\frac M2-\frac m2+\frac 12},-\alpha q^{\frac M2-\frac m2+\frac 12}\\
\alpha^2,\alpha\beta  q^{\frac M2+\frac N2},-\alpha\beta q^{\frac M2-\frac N2}\end{matrix};q,q\right)
\\
\begin{split}&= {(-1)^{\binom{|N|}2}\varepsilon^{\binom{|M|}2+N(n+M)}\alpha^{(2m-|M|)n+\binom {|N|}2}\beta^{n|N|+\binom{|M|+|N|}2+\binom{|N|}2}}q^X\\
&\quad\times
\prod_{j=1}^{|M|}\frac{(-\varepsilon\alpha\beta q^{j+n-1-\frac{|M|}2-\frac{|N|}2};q)_{|N|}}{(q,-q^{j+1-|M|}/\alpha^{2};q)_{j-1}(-\varepsilon q^{j-\frac{|M|}2-\frac{|N|}2}\beta/\alpha;q)_{|N|}}\\
&\quad\times\prod_{j=1}^{m/2}\frac{(\alpha^2\beta^2q^{2j-3-|M|-|N|};q^2)_{\left[\frac{|M|+|N|+n+1}2\right]}(\alpha^2\beta^2q^{2j-1-|M|-|N|};q^2)_{\left[\frac{|M|+|N|+n}2\right]}}{(\alpha^2\beta^2q^{2j-3-|M|-|N|};q^2)_{|M|+|N|+n}(\alpha^2\beta^2q^{2j-1-|M|-|N|};q^2)_{n}}\\
&\quad\times \prod_{j=1}^m\frac{(\alpha^2\beta^2q^{n-1};q)_{n+j-1}(q,q^{1-m-n};q)_{j-1}(\alpha^2q^{M-m+1};q^2)_{j-1}}
{(\alpha^2,\alpha\beta  q^{\frac{M+N}2},-\alpha\beta q^{\frac{M-N}2};q)_{n+j-1}}\\
&\quad\times\prod_{j=1}^{n}\frac{
(q^{2[j/2]+1},q^{-2[(j-1)/2]+1-m+|M|}/\alpha^2,\beta^2q^{-|N|+2[(j-1)/2]+1};q^2)_{\frac m2}
}{(q^{j+1-n}/\alpha^2;q)_{|M|}(\beta^2 q^{j-1-|N|};q)_{|N|}}\\
&\quad\times\prod_{j=1}^{|M|+|N|}\frac 1{( \alpha^2\beta^2q^{j+2n-2-|M|-|N|};q)_{j-1}}
\prod_{j=1}^{|N|}\frac 1{(q,-q^{j-|N|-1}\beta^{2};q)_{j-1}}\\
&\quad\times Q^{|M|,|N|,n}(\alpha q^{-\frac{|M|}2},-\varepsilon\beta q^{-\frac{|N|}2},q^{\frac{m+1}2}),
\end{split}\end{multline*}
}
where
\begin{align*}X&=2\binom m3+\binom m2-\frac 14|N|\big(|M|^2+|MN|+2|N|-2\big)
\\
&\quad+\frac 12\big\{m\left(m+M-1\right)+|M|(|M|+1-m-n)-|N|\big\}n\\
&\quad+\frac 12\left(\binom{|M|+|N|+n}3-\binom n3\right)\\
&\quad+\frac m2\left(\left[\frac{(|M|+|N|+n-1)^2}4\right]+\left[\frac{(n-1)^2}4\right]\right)
.\end{align*}
\end{corollary}

\subsection{Final steps}

We can now complete the proof of Theorem \ref{mt}.
Assume first that $m$ is even. Combining Corollary \ref{pdc}, Proposition \ref{dsp} and Corollary \ref{fpc}, with the substitutions
$$(\alpha,\beta,n)\mapsto(q^{\frac 12(1-b-c-m)},q^{\frac 12(1+a+c+m)},b), $$
 yields
\eqref{mti}, with
{
\allowdisplaybreaks
\begin{align*}C&=\frac {(-1)^{\binom{|N|}2}\varepsilon^{\binom{|M|}2+N(n+M)}q^X}{2^{\frac 12{m(a+b+M+N)}+ab}}\prod_{j=1}^b\frac{(q;q)_{a+c+m+j-1}(q;q)_{j-1}(-q^{\frac 12(-a-b+M-N)+j};q)_{a}}{(q;q)_{a+j-1}(q;q)_{c+m+j-1}}\\
&\quad\times
\prod_{j=1}^m\frac{(q;q)_{j-1}(q^2;q^2)_{\frac 12(a+c+N)+j-1}}{(q;q)_{\frac{a+b+M+N}2+j-1}(q^2;q^2)_{\frac{b+c-M}2+j-1}}\\
&\quad\times
\prod_{j=1}^{|M|}\frac{(-\varepsilon q^{j+\frac a2+\frac b2-\frac{|M|}2-\frac{|N|}2};q)_{|N|}}{(q,-q^{j+b+c-|M|+m};q)_{j-1}(-\varepsilon q^{j+\frac a2+\frac b2+c+m-\frac{|M|}2-\frac{|N|}2};q)_{|N|}}\\
&\quad\times\prod_{j=1}^{m/2}\frac{(q^{2j-1+a-b-|M|-|N|};q^2)_{\left[\frac{|M|+|N|+b+1}2\right]}(q^{2j+1+a-b-|M|-|N|};q^2)_{\left[\frac{|M|+|N|+b}2\right]}}{(q^{2j-1+a-b-|M|-|N|};q^2)_{|M|+|N|+b}(q^{2j+1+a-b-|M|-|N|};q^2)_{b}}\\
&\quad\times\prod_{j=1}^{b}\frac{
(q^{2[j/2]+1},q^{-2[(j-1)/2]+b+c+|M|},q^{a+c+m-|N|+2[(j-1)/2]+2};q^2)_{\frac m2}
}{(q^{j+c+m};q)_{|M|}( q^{j+a+c+m-|N|};q)_{|N|}}\\
&\quad\times\prod_{j=1}^{|M|+|N|}\frac 1{(q^{j+a+b-|M|-|N|};q)_{j-1}}
\prod_{j=1}^{|N|}\frac 1{(q,-q^{j+a+c+m-|N|};q)_{j-1}}
\end{align*}
}
for a certain exponent $X$.

We now express $C$ in terms of the function $\tilde H_q$. The product $\prod_{j=1}^{m/2}$ is equal to
\begin{multline}\label{mhp}\prod_{j=1}^{m/2}\frac{(q;q^2)_{[(a+1)/2]+j-1}(q;q^2)_{[a/2]+j}}{(q;q^2)_{(a+b+|M|+|N|)/2+j-1}(q;q^2)_{(a+b-|m|-|N|)/2+j}}\\
=\frac{\tilde H_{q^2}\left(\left[\frac{a+1}{2}\right]+\frac{m-1}2,\left[\frac{a}{2}\right]+\frac{m+1}2,\frac{a+b+|M|+|N|-1}2,\frac{a+b-|M|-|N|+1}2\right)}{\tilde H_{q^2}\left(\left[\frac{a+1}{2}\right]-\frac{1}2,\left[\frac{a}{2}\right]+\frac{1}2,\frac{a+b+|M|+|N|+m-1}2,\frac{a+b-|M|-|N|+m+1}2\right)}. \end{multline}
All other factors can be converted using Lemma \ref{phl} and Lemma \ref{phlb}. The end result can then be simplified  using 
 Lemma \ref{qhl}. For instance, the factor $\tilde H_q(a)$ appearing from $\prod_{j=1}^b(q;q)_{a+j-1}^{-1}$ combines with factors from \eqref{mhp} as
$$\frac{\tilde H_q(a)}{\tilde H_{q^2}\left(\left[\frac{a+1}{2}\right]-\frac{1}2,\left[\frac{a}{2}\right]+\frac{1}2\right)}=\tilde H_{q^2}\left(\left[\frac a2\right],\left[\frac{a+1}2\right]\right). $$
We also observe that
\begin{align*}\prod_{j=1}^{|M|}(-\varepsilon q^{j+\frac{a+b-|M|-|N|}2};q)_{|N|}&=\frac{\tilde H_q^{-\varepsilon}\left(\frac{a+b-|M|-|N|}2,\frac{a+b+|M|+|N|}2\right)}{\tilde H_q^{-\varepsilon}\left(\frac{a+b+|M|-|N|}2,\frac{a+b-|M|+|N|}2\right)}\\
&=\frac{\tilde H_{q^2}\left(\frac{a+b-|M|-|N|}2,\frac{a+b+|M|+|N|}2\right)}{\tilde H_q\left(\frac{a+b-M-N}2,\frac{a+b+M+N}2\right)\tilde H_q^-\left(\frac{a+b+M-N}2,\frac{a+b-M+N}2\right)}, \end{align*}
where most factors on the right  eventually  cancel.
Applying Lemma \ref{phlb} to 
$$\prod_{j=1}^b(-q^{\frac 12(-a-b+M-N)+j};q)_a $$
leads, apart from functions $\tilde H_q^-$ and a power of $q$,
to the factor
$$D=\frac{D_{(a-b+M-N)/2}D_{(-a+b+M-N)/2}}{D_{(-a-b+M-N)/2}D_{(a+b+M-N)/2}}, $$
where $D_k=2^{\min(k,0)}$. Since $-a-b+M-N<0$ and $a+b+M-N>0$, it is easy to see that
$D=2^{\frac{a+b}2-\frac 12\max(|a-b|,|M-N|)}$. 
Finally, we express $\tilde H_q$ in terms of $H_q$. In this way, we find that
 \eqref{mti} holds up to some factor $q^X$, where $X$ is independent of $q$.  Since all functions in \eqref{mti} are 
 invariant  up to a sign
under replacing $q$ by $q^{-1}$ (here we use \eqref{qdi}), we must have $X=0$. 
This proves Theorem \ref{mt} in the case when $m$ is even.

If $m$ is odd we invoke Lemma \ref{pl}.
 Since a rational function
is determined by infinitely many values, 
 it is enough to prove that the right-hand side of \eqref{mti}, considered as a function of $m$, has the same form as the right-hand side of \eqref{pli}.
Since \eqref{qd} is a Laurent polynomial in $\alpha,\beta,\gamma$, 
the second factor in \eqref{mti} is rational in $q^{m/2}$. 
The final factor in \eqref{mtc}, involving $H_q^{-\varepsilon}$, is rational in $q^m$ by Lemma \ref{qhl}. 
Consider now the remaining $q$-hyperfactorials in \eqref{mtc}.
We rewrite $H_q^-$ as $H_{q^2}/H_q$ and then apply Lemma \ref{qhl} to  all factors of the
form $f(x+m)$, where $f=H_q$ or $H_{q^2}$. This leads to a product of the form
$$\prod_{j=1}^{6}\frac{H_{q^2}\left (a_j+\frac{m}2\right)}{H_{q^2}\left (b_j+\frac{m}2\right)}\prod_{j=1}^{10}\frac{H_{q^2}\left (c_j+\frac{m-1}2\right)}{H_{q^2}\left (d_j+\frac{m-1}2\right)}\prod_{j=1}^{10}\frac{H_{q^4}\left (e_j+\frac{m}2\right)}{H_{q^4}\left (f_j+\frac{m}2\right)}\prod_{j=1}^{10}\frac{H_{q^4}\left (g_j+\frac{m-1}2\right)}{H_{q^4}\left (h_j+\frac{m-1}2\right)}, $$  
where $a_1,\dots,h_{10}$ are all integers. One may check that
\begin{multline*}\sum_{j=1}^6 (b_j-a_j)=\sum_{j=1}^{10}(d_j-c_j)=\sum_{j=1}^{10}(e_j-f_j)=\sum_{j=1}^{10}(g_j-h_j)\\
=\frac{a+b+M+N}2=a+b-C.\end{multline*}
Applying Lemma \ref{mhl}, it follows that the right-hand side of \eqref{mti} indeed behaves as \eqref{zw} as a function of $m$. This completes the proof
of Theorem \ref{mt}.

\section*{Appendix. The $q$-hyperfactorial}

 \renewcommand{\thesection}{\Alph{section}}
 \setcounter{equation}{0}
\setcounter{theo+}{0}
\setcounter{section}{1}

In this Appendix, we collect some elementary properties of
 the $q$-hyperfactorials $H_q$ and  $\tilde H_q$ defined in \eqref{qsf}.

\begin{lemma}\label{hrl}
Let $a_1,\dots,a_n,b_1,\dots,b_n$ be non-negative integers, such that
$\sum a_j=\sum  b_j$ and let $H$ be any one of the
functions $H_q$, $H_q^-$, $\tilde H_q$ or $\tilde H_q^-$.
Then, there exists a rational function $f$ such that
\begin{equation}\label{hfr}\prod_{k=1}^n\frac{H\left(a_k+\frac m2\right)}{H\left(b_k+\frac m2\right)}=f(q^{m/2}) \end{equation}
for each non-negative integer $m$. \end{lemma}

\begin{proof}
It is easy to check that
\begin{multline}\label{hqd}\tilde H_q\left(a+\frac m2\right)\\
=\begin{cases}\tilde H_q(m/2)(q;q)_{m/2}^a \prod_{j=1}^a(q^{1+m/2};q)_{j-1}, & m \text{ even},\\
\tilde H_q((m+1)/2)(q^{1/2};q)_{(m+1)/2}^a \prod_{j=1}^a(q^{1+m/2};q)_{j-1}, & m \text{ odd}.
\end{cases} \end{multline}
It follows that \eqref{hfr} holds for $H=\tilde H_q$, with
$$f(x)=\prod_{k=1}^n\frac{\prod_{j=1}^{a_k}(qx;q)_{j-1}}{\prod_{j=1}^{b_k}(qx;q)_{j-1}}. $$

Let us now replace $\tilde H_q$ by $H_q$. If $m$ is even,
\eqref{hfr} is multiplied with $q^{Q(m)/2}$, where $Q$ is the polynomial
$$Q(m)=\sum_{k=1}^n\left(\binom{b_k+\frac m2+1}{3}-\binom{a_k+\frac m2+1}{3}\right). $$
Note that the cubic and quadratic terms in $Q$ cancel, so $q^{Q(m)/2}$ is a rational function in $q^{m/2}$. Moreover, since $\frac 1 8\binom{2m+1}3=\binom{m+1}3+\frac m8 $, the same multiplier appears when $m$ is odd. This proves the case $H=H_q$.  The remaining two cases follow since $H_q^-=H_{q^2}/H_q$
and  $\tilde H_q^-=\tilde H_{q^2}/\tilde H_q$.
\end{proof}

We will also need the following variation of Lemma \ref{hrl}.

\begin{lemma}\label{mhl}
Let $a_j$, $b_j$, $c_j$ and $d_j$ be non-negative integers such that
$$\sum_{j=1}^k(a_j-b_j)=\sum_{j=1}^l(c_j-d_j)=\lambda.$$
 Then there exists a rational function $f$ such that
$$\prod_{j=1}^k\frac{H_q\left(a_j+\frac m2\right)}{H_q\left(b_j+\frac m2\right)}
\prod_{j=1}^l\frac{H_q\left(c_j+\frac {m-1}2\right)}{H_q\left(d_j+\frac {m-1}2\right)}
=\left(q^{-\frac 14\binom{m+1}2}(q^{1/2};q^{1/2})_m\right)^\lambda f(q^{m/2})
 $$
for each non-negative integer $m$.
\end{lemma}

\begin{proof}
It follows  from \eqref{hqd} that
$$\prod_{j=1}^k\frac{H_q\left(a_j+\frac m2\right)}{H_q\left(b_j+\frac m2\right)}
=\begin{cases}\left(q^{-\frac 1{16}(m+2)m}(q;q)_{m/2}\right)^\lambda f(q^{m/2}), & m \text{ even},\\
\left(q^{-\frac 1{16}(m+1)^2} (q^{1/2};q)_{(m+1)/2}\right)^\lambda f(q^{m/2}), & m \text{ odd},
\end{cases}$$
where $f$ is rational. Replacing $m$ by $m-1$, it follows that
$$\prod_{j=1}^l\frac{H_q\left(c_j+\frac {m-1}2\right)}{H_q\left(d_j+\frac {m-1}2\right)}
=\begin{cases}\left(q^{-\frac 1{16}m^2} (q^{1/2};q)_{m/2}\right)^\lambda g(q^{m/2}),& m \text{ even}\\
\left(q^{-\frac 1{16}(m+1)(m-1)}(q;q)_{(m-1)/2}\right)^\lambda g(q^{m/2}), & m \text{ odd},
\end{cases}$$
again with $g$ rational. Using
$$(q^{1/2};q^{1/2})_m=\begin{cases}
(q^{1/2};q)_{m/2}(q;q)_{m/2}, & m \text{ even},\\
(q^{1/2};q)_{(m+1)/2}(q;q)_{(m-1)/2},
 & m \text{ odd}
\end{cases}$$
we obtain the desired result.
\end{proof}

The following two lemmas are straight-forward to verify.

\begin{lemma}\label{qhl}
For $m$  a  non-negative integer,
$$\tilde H_q(m)=\tilde H_{q^2}\left(\frac{m-1}2,\frac m2,\frac m2,\frac{m+1}2\right). $$ 
The same identity holds if $\tilde H_q$ is replaced by $H_q$.
\end{lemma}

\begin{lemma}\label{phl}
For $k$, $l$ and $m$ non-negative integers,
\begin{align*}
\prod_{j=1}^k(q;q)_{m+j-1}&=\frac{\tilde H_q(k+m)}{\tilde H_q(m)},\\
\prod_{j=1}^k(q^{l+j};q)_{m}&=\frac{\tilde H_q(l,k+l+m)}{\tilde H_q(k+l,l+m)},\\
\prod_{j=1}^k(q^{l+j};q)_{j-1}&=\frac{ \tilde H_{q^2}\left(\frac l2,\frac{l+1}2,\frac{l-1}{2}+k,\frac{l}{2}+k\right)}{\tilde H_q(l+k)}.\end{align*}
Moerover, for $k$, $2l+1$ and $m$ non-negative integers,
\begin{align*}\prod_{j=1}^k(q^{l+1+[(j-1)/2]};q)_m&=\frac{\tilde H_q(l,l,l+m+[k/2],l+m+[(k+1)/2])}{\tilde H_q(l+m,l+m,l+[k/2],l+[(k+1)/2])}, \\
\prod_{j=1}^k(q^{l+1+[j/2]};q)_m&=\frac{\tilde H_q(l,l+1,l+m+[(k+1)/2],l+m+[(k+2)/2])}{\tilde H_q(l+m,l+m+1,l+[(k+1)/2],l+[(k+2)/2])},\\
\prod_{j=1}^k(q^{l+\frac k2-\left[\frac{j-1}2\right]};q)_m&=\frac{\tilde H_q\left(l-\varepsilon,l+\varepsilon,l+m+k/2,l+m+k/2\right)}{\tilde H_q\left(l+m-\varepsilon,l+m+\varepsilon,l+k/2,l+k/2\right)},
\end{align*}
where $\varepsilon=0$ for $k$ even and $\varepsilon=1/2$ for $k$ odd.
\end{lemma}

We  also need the following variation of the second identity in
Lemma \ref{phl}.

\begin{lemma}\label{phlb}
For $k$ and $m$ non-negative integers and $l$ an arbitrary integer,
$$\prod_{j=1}^k(-q^{l+j};q)_m=\frac{C_{k+l}C_{l+m}}{C_lC_{k+l+m}}\frac{\tilde H_q^-(|l|,|k+l+m|)}{\tilde H_q^-(|k+l|,|l+m|)}, $$
where 
$C_n=1$ for $n\geq 0$ and $C_n=2^{n}q^{(n-n^3)/6}$ for $n<0$.
\end{lemma}

\begin{proof}
By induction on $k$, the result is reduced to
$$(-q^{k+l+1};q)_m=\frac{C_{k+l+1}C_{k+l+m}}{C_{k+l}C_{k+l+m+1}}\frac{\tilde H_q^-(|k+l+m+1|,|k+l|)}{\tilde H_q^-(|k+l+m|,|k+l+1|)}. $$
By induction on $m$, this is in turn reduced to
$$1+q^{k+l+m+1}=\frac{C_{k+l+m+1}^2}{C_{k+l+m}C_{k+l+m+2}}\frac{\tilde H_q^-(|k+l+m|,|k+l+m+2|)}{\tilde H_q^-(|k+l+m+1|)^2}. $$
It is easy to check that
$$\frac{\tilde H_q^-(|n+1|,|n-1|)}{\tilde H_q^-(|n|)^2}=\begin{cases}1+q^n, & n\geq 1,\\
1, & n=0,\\
1+q^{-n}, & n\leq -1.\end{cases} $$
Thus, the result holds for any solution  to the recursion
$$\frac{C_n^2}{C_{n-1}C_{n+1}}=\begin{cases}1, & n\geq 1,\\ 2, & n=0,\\ q^n, & n\leq -1.\end{cases} $$
The given solution corresponds to the initial values $C_0=C_1=1$. 
\end{proof}


\begin{thebibliography}{99}

\bibitem[AW]{aw} R.\ Askey and J.\ Wilson, \emph{Some basic hypergeometric orthogonal polynomials that generalize Jacobi polynomials}, Mem.\ Amer.\ Math.\ Soc.\ 54 (1985),  319.

\bibitem[BGR]{bgr} A.\ Borodin, V.\ Gorin and E.\ M.\ Rains, \emph{$q$-Distributions on boxed plane partitions}, Selecta Math.\ 16 (2010), 731--789. 

\bibitem[BH]{bh} E.\ Br\'ezin and S.\ Hikami, \emph{Characteristic polynomials of random matrices}, Comm.\ Math.\ Phys.\ 214 (2000), 111--135.

\bibitem[C1]{c1} M.\ Ciucu, 
\emph{Enumeration of lozenge tilings of punctured hexagons},
J.\ Combin.\ Theory Ser.\ A 83 (1998),  268--272. 

\bibitem[C2]{c2} M.\ Ciucu, 
\emph{Plane partitions I: A generalization of MacMahon's formula}, Mem.\  Amer.\ Math.\ Soc.\ 178 (2005), no.\ 839, 107--144.


\bibitem[C3]{c3} M.\ Ciucu, 
\emph{A generalization of Kuo condensation}, arXiv:1404.5003.

\bibitem[CEKZ]{cekz} 
M.\ Ciucu, T.\ Eisenk\"olbl, C.\ Krattenthaler and D.\ Zare,
\emph{Enumeration of lozenge tilings of hexagons with a central triangular hole}, J.\ Combin.\ Theory  A 95 (2001),  251--334.

\bibitem[CF1]{cf} M.\ Ciucu and I.\ Fischer,
\emph{Proof of two conjectures of Ciucu and Krattenthaler on the enumeration of lozenge tilings of hexagons with cut off corners}, arXiv:1309.4640. 

\bibitem[CF2]{cf2} M.\ Ciucu and I.\ Fischer,
\emph{Lozenge tilings of hexagons with arbitrary dents}, arXiv:1412.3945.






\bibitem[CK1]{ck1} M.\ Ciucu and C.\ Krattenthaler,
 \emph{Plane partitions. II. 5\,1/2 symmetry classes}, Combinatorial Methods in Representation Theory, pp.\ 81--101,  Kinokuniya,  2000. 

\bibitem[CK2]{ck2} M.\ Ciucu and C.\ Krattenthaler,
 \emph{Enumeration of lozenge tilings of hexagons with cut-off corners}, J.\ Combin.\ Theory  A 100 (2002),  201--231. 

\bibitem[CK2]{ck3} M.\ Ciucu and C.\ Krattenthaler,
 \emph{A dual of MacMahon's theorem on plane partitions}, Proc.\ Natl.\ Acad.\ Sci.\ USA 110 (2013),  4518--4523.

\bibitem[CK4]{ck4} M.\ Ciucu and C.\ Krattenthaler,
 \emph{A factorization theorem for lozenge tilings of a hexagon with triangular holes}, arXiv:1403.3323.


\bibitem[E1]{e1} T.\ Eisenk\"olbl, 
\emph{Rhombus tilings of a hexagon with two triangles missing on the symmetry axis},
Electron.\ J.\ Combin.\ 6 (1999), Research Paper 30.

\bibitem[E2]{e2} T.\ Eisenk\"olbl, 
\emph{Rhombus tilings of a hexagon with three fixed border tiles},
J.\ Combin.\ Theory  A (1999),  368--378. 

\bibitem[FW]{fw} P.\ J.\ Forrester and S.\ O.\ Warnaar, \emph{The importance of the Selberg integral}, Bull.\ Amer.\ Math.\ Soc.\  45 (2008),  489--534.

\bibitem[GR]{gr}  G.\ Gasper and M.\ Rahman,  Basic Hypergeometric Series,
$2^{\text{nd}}$ ed., Cambridge University Press, 2004.



\bibitem[GV]{gv} I.\ Gessel and G.\ Viennot,
\emph{Determinants, paths, and plane partitions},
preprint, 1989.

\bibitem[HG]{hg} H.\ A.\ Helfgott and I.\ M.\  Gessel, 
\emph{Enumeration of tilings of diamonds and hexagons with defects},
Electron.\ J.\ Combin.\ 6 (1999), Research Paper 16. 

\bibitem[I]{i} M.\ E.\ H.\ Ismail, Classical and Quantum Orthogonal Polynomials in One Variable, Cambridge University Press, Cambridge, 2005. 

\bibitem[K1]{kr} C.\ Krattenthaler, \emph{Generating functions for plane partitions of a given shape}, Manuscripta Math.\ 69 (1990), 173--202.

\bibitem[K2]{kad} C.\ Krattenthaler, \emph{Advanced determinant calculus},
 S\'em.\ Lothar.\ Combin.\ 42 (1999), B42q.

\bibitem[K3]{kr2} C.\ Krattenthaler, \emph{Advanced determinant calculus, a complement}, Linear Algebra Appl.\ 411 (2005), 68--166.

 \bibitem[L1]{l1} T.\ Lai, \emph{A new proof for a generalization of a Proctor's formula on plane partitions}, arXiv:1410.8116. 

 \bibitem[L2]{l2} T.\ Lai, \emph{Enumeration of lozenge tilings of
 a hexagon with shamrock hole on boundary}, arXiv:1502.01679.

 \bibitem[L3]{l3} T.\ Lai, \emph{Enumeration of lozenge tilings of
  hexagon with three holes}, arXiv:1502.05780.




\bibitem[Li]{l} B.\ Lindstr\"om, \emph{On the vector representation of  induced matroids}, Bull.\ London Math.\ Soc.\ 5 (1973), 85--90.

\bibitem[M]{mm} P.\ A.\ MacMahon, Combinatory Analysis, Vol.\ II, Cambridge University Press, 1916.

\bibitem[OK]{ok} S.\ Okada and C.\ Krattenthaler,
\emph{The number of rhombus tilings of a "punctured'' hexagon and the minor summation formula},
Adv.\ Appl.\ Math.\ 21 (1998),  381--404. 

\bibitem[P]{pro} R.\ A.\ Proctor, 
\emph{Odd symplectic groups},
Invent.\ Math.\ 92 (1988),  307--332. 

\bibitem[Pr]{pr} J.\ Propp, 
\emph{Enumeration of matchings: problems and progress}, New Perspectives in Algebraic Combinatorics, pp.\ 255--291, Cambridge Univ.\ Press,  1999. 

\bibitem[Ra]{ra} E.\ M.\ Rains, \emph{$BC_n$-symmetric abelian functions}, Duke Math.\ J.\ 135 (2006),  99--180. 


\bibitem[S1]{sc1} M.\ Schlosser, \emph{Summation theorems for multidimensional basic hypergeometric series by determinant evaluations},  Discrete Math.\  210  (2000),   151--169. 

\bibitem[S2]{sc2} M.\ Schlosser, \emph{Elliptic enumeration of nonintersecting lattice paths},
J.\ Combin.\ Theory Ser.\ A 114 (2007),  505--521. 
 
\bibitem[W]{w} S.\ O.\ Warnaar,
\emph{Summation and transformation formulas for elliptic hypergeometric series}, Constr.\ Approx.\ 18 (2002),  479--502. 
\vskip 3mm
\end{thebibliography}
 \end{document}
 \end